\documentclass[11pt]{amsart}

\usepackage{amsmath, amsthm, amssymb}
\usepackage{graphicx, array}
\usepackage{longtable}
\usepackage{epstopdf}
\usepackage{mathrsfs}

\theoremstyle{definition}

\newtheorem{theorem}{Theorem}[section]
\newtheorem{proposition}{Proposition}
\newtheorem{lemma}{Lemma}

\newtheorem{remark}{Remark}
\newtheorem{Ex}{Example}

\title{The word problem for some classes of Adian inverse semigroups-II}

\author{Muhammad Inam}
\email{m-inam@onu.edu}

\address{School of Science, Technology \& Mathematics,\\
Ohio Northern University\\
Ada, Ohio 45810 USA}

\begin{document}
\date{\today}

\begin{abstract} We introduce the notion of a subgraph generated by an $R$-word $r$ of the Sch\"{u}tzenberger graph of a positive word $w$, $S\Gamma(w)$, where $w$ contains $r$ as its subword. We show that the word problem for a finitely presented Adian inverse semigroup $Inv\langle X|R \rangle$ is decidable if the subgraphs of $S\Gamma(t)$, for all $t\in X^+$, generated by all the $R$-words over the presentation $\langle X|R\rangle$, are finite. As a consequence of this result, we show that the word problem is decidable for some classes of one relation Adian inverse semigroups.

\end{abstract}

\maketitle

\section{Introduction}

\bigskip

Throughout this paper $X$ denotes an alphabet. The set $R=\{ (u_i,v_i)|i\in I\}$, where $u_i,v_i\in X^+$, denotes the set of \textit{positive relations}. The words $u_i$ and $v_i$ are called \textit{$R$-words}, provided that $(u_i,v_i)\in R$.  The pair $\langle X|R\rangle$ is called a \textit{positive presentation}. The  semigroup generated by the set $X$ and having a set $R$ of relations is denoted by $Sg\langle X|R\rangle$, and the group generated by the set $X$ and having a set $R$ of relations is denoted by $Gp\langle X|R\rangle$. There exists a natural homomorphism $\phi$:$Sg\langle X|R\rangle$ $\to$ $Gp\langle X|R\rangle$. 

We can construct two undirected graphs corresponding to a positive presentation. The \textit{left graph} of the presentation $\langle X|R\rangle$ is  denoted by $LG\langle X|R\rangle$. The vertices of $LG\langle X|R\rangle$ are labeled by the elements of $X$, and there is an edge corresponding to every relation $(u_i,v_i)\in R$, that connects the first (prefix) letters of $u_i$ and $v_i$ together. Similarly, the \textit{right graph} of the presentation $\langle X|R\rangle$ is denoted by $RG\langle X|R\rangle$, and it can be obtained by connecting the last (suffix) letters of $u_i$ and $v_i$ together, for every $(u_i,v_i)\in R$.  A closed path in  $LG\langle X|R\rangle$ is called a \textit{left cycle} and a closed path in $RG\langle X|R\rangle$ is called a \textit{right cycle}. Further details about the left and right graphs of a positive presentation $\langle X|R\rangle$ along with some examples can be found in \cite{Thesis}. If for some presentation $\langle X|R\rangle$, there is no closed path (cycle) in $LG\langle X|R\rangle$ and $RG\langle X|R\rangle$, then the presentation is called a \textit{cycle free} presentation. A cycle free presentation is also called an Adian presentation  because these presentations were first studied by S. I. Adian \cite{Adian}. A semigroup (group) presented by an Adian presentation is called an \textit{Adian semigroup (Adian group)}.  

A semigroup $S$ is called an \textit{inverse semigroup} if for every element $a\in S$ there exists a unique element $b\in S$ such that $aba=a$ and $bab=b$.  The unique element $b$ is denoted by $a^{-1}$. The idempotents commute in an inverse semigroup, and the product of two idempotents is an idempotent. The \textit{natural partial order} on an inverse semigroup $S$ is defined as $a\leq b$ if and only if $aa^{-1}b=a$, for some $a,b\in S$.  A congruence relation $\sigma$ is defined on $S$, for $a,b\in S$, by $a\sigma b$ if and only if there exists an element $c\in S$ such that $c\leq a,b$. It turns out that $\sigma$ is the minimum group congruence relation on $S$, i.e., $S/\sigma$ is the maximum group homomorphic image of $S$. Just like groups and semigroups, an inverse semigroups can also be presented by a set of generators and a set of relations. We denote an inverse semigroup by $Inv\langle X|R\rangle$, where $X$ is the set of generators, and $R$ is the set of relations. An inverse semigroup presented by an Adian presentation is called an \textit{Adian inverse semigroup}. If $S=Inv\langle X|R\rangle$, then $S/\sigma$ is the group $Gp\langle X|R\rangle$. The set of idempotents of $S$ is denoted by $E(S)=\{e\in S: e^2=e\}$.  An inverse semigroup is called \textit{E-unitary} if the inverse image of $1$(the identity element of the group $S/\sigma$) is precisely $E(S)$. All these facts about the inverse semigroups along with more details can be found in the text \cite{Lawson}.  

 In \cite{RG} Gray proved that the word problem for one relator $E$-unitary inverse semigroups is undecidable in general. 

Adian conjectured \cite{Adian} that the word problem is decidable for Adian semigroups. The following result was first proved by Adian \cite{Adian} for only finite Adian presentations. Later Remmers \cite{Remmers} proved the same result by using geometric techniques for any Adian presentation. 

\begin{theorem} An Adain semigroup embeds in the Adian group with the same presentation. 

\end{theorem}

Magnus \cite{Lyndon} proved that the word problem is decidable for one relator groups. Since Adian semigroups embed in Adian groups, it follows that the word problem is decidable for one relation Adian semigroups as well. However, the question whether the word problem is decidable for one relation Adain inverse semigroups or not, has not been answered yet. 

The following theorem is proved in \cite{E-unitary}
\begin{theorem}\label{E-unitary} Adian inverse semigroups are $E$-unitary.
\end{theorem}

By using Theorem \ref{E-unitary}, we can make the following remark about one relation Adian inverse semigroups. 

\begin{remark} Let $M=Inv\langle X|u=v\rangle$ be an Adian inverse semigroup. Then the membership problem for the set of idempotents of $M$, $E(M)$, is decidable.
\end{remark}

The following result has been proved in \cite{AD}

\begin{theorem} \label{PMT} Let $M=Inv\langle X|R\rangle$ be a finitely presented Adian inverse semigroup. Then the Sch\"{u}tzenberger graph of $w$, for all $w\in (X\cup X^{-1})^+$, is finite if and only if the Sch\"{u}tzenberger graph of $w'$ is finite, for all $w'\in X^+$. 
\end{theorem}

In this paper, we prove the following result. 

\begin{theorem}\label{MT2}  Let $M=Inv\langle X|R\rangle$ be a finitely presented Adian inverse semigroup.  The Sch\"{u}tzenberger graph of every positive word  is finite if and only if the subgraphs of $S\Gamma(w)$, for all $w\in X^+$, generated by all the $R$-words are finite.

\end{theorem}

It follows from Theorem \ref{PMT} and \ref{MT2} that in a finitely presented Adian inverse semigroup $M=Inv\langle X|R\rangle$, if  the subgraphs of the Sch\"{u}tzenberger graph of $w'$, for all $w'\in X^+$, generated by all the $R$-words, are finite, then the Sch\"{u}tzenberger graph of $w$, for all $w\in (X\cup X^{-1})^+$, is finite, which implies the decidability of the word problem for $M$.

In this paper we also study the word problem for one relation Adian inverse semigroups and show that the word problem is decidable for the following classes of one relation Adian inverse semigroups.

\begin{theorem}\label{MT1} Let $M=Inv\langle X|u=v\rangle$ be an Adian inverse semigroup, such that no $R$-word is a subword of the other $R$-word, and the relation $(u,v)$ is in one of the following forms:
\begin{enumerate} 

\item No $R$-word overlaps with itself or with the other $R$-word.

\item One of the $R$-words overlaps with itself, and the other $R$-word neither overlaps with itself nor with the former $R$-word. 

\item Both $R$-words overlap with themselves, there is no overlap between both the $R$-words, and at least one of the $R$-words is not of the form $x^n$, for some $x\in X^+$ and $n\geq 2$.  

\item Prefix of one $R$-word is a suffix of the other $R$-word, no suffix of the former $R$-word is a prefix of  the latter $R$-word, and no $R$-word overlaps with itself. 
\end{enumerate}

Then the word problem is decidable for $M$. \end{theorem}








\section{Outline} 

In this paper, we study the word problem for  a finitely presented Adian inverse semigroup. 

Section $3$ of this paper provides some fundamental definitions and results of inverse semigroup theory that have been used in this paper. In this section, the notion of Sc\"{u}tzenberger graphs has been elaborated and an iterative procedure for constructing these graphs is provided. 

Section $4$ of this paper mainly consists of the proof of Theorem \ref{MT2}. At the beginning of this section we introduce the notion of a subgraph of a Sch\"{u}tzenberger graph, generated by an $R$-word. We provide an iterative procedure for constructing these subgraphs. We also provide some basic structural properties of these subgraphs. We give the definition of $S$-diagram as it has been used in the proof of Theorem \ref{MT2}. 

Section $5$ of this paper provides some of the applications of Theorem \ref{MT2}. In this section we study the word problem for one relation Adian inverse semigroups, and prove Theorem \ref{MT1}.

\section{Preliminaries}

For an alphabet $X$, the set $X^+$ denotes the set of all positive words on $X$, and the set $X^*$ denotes the set $X^+\cup \{\epsilon\}$, where $\epsilon$ is the empty word. The set $X^+$ forms a semigroup under the binary operation of concatenation of words, and called the free semigroup on $X$. The empty word $\epsilon$ serves as the identity element under the concatenation of words, and the set $X^*$ denotes the free monoid on $X$. If $\rho$ denotes the congruence generated by a set relations $R$, then $X^+/\rho$ is the \textit{semigroup} given by the presentation $Sg\langle X|R\rangle$. For $w_1, w_2\in X^*$, we write $w_1\equiv w_2$ when $w_1$ and $w_2$ are identical words, and write $w_1=w_2$ when  $w_1\rho=w_2\rho$ in the monoid $Mon\langle X|R\rangle := X^*/\rho$, where $\rho$ is the congruence generated by $R$. If $\rho$ denotes the Vagner congruence on the monoid $(X\cup X^{-1})^*$, then $FIM(X):=(X\cup X^{-1})^*/\rho$ is the \textit{free inverse monoid} on $X$. If $R\subseteq (X\cup X^{-1})^*\times (X\cup X^{-1})^*$, and $\tau$ denotes the congruence relation generated by $R\cup \rho$, then $M=Inv\langle X|R \rangle := (X\cup X^{-1})^*/\tau$ is the inverse monoid presented by the set of generators $X$ and the set of relations $R$. 

A \textit{labeled directed graph} over a set $X$ is a directed graph in which the edges are labeled by elements of $X$. We write $(u,x,v)$ to denote the edge labeled by $x$ with initial vertex $u$ and terminal vertex $v$. A \textit{path} or \textit{segment} of length $n$ is a sequence of edges 

$\{(v_0,x_1,v_1), (v_1, x_2, v_2),..., (v_{n-1}, x_n,v_n)\}$ 

such that the initial vertex of an edge (except the first) equals the terminal vertex of the previous edge. if $v_0=v_n$, the path is a cycle. We say that the path is labeled by the word $w=x_1 x_2....x_n$ and that $w$ can be read in the graph starting at $v_0$.

An \textit{inverse word graph} over $X$ is a labeled directed graph over $X\cup X^{-1}$ such that the labeling is consistent with an involution, that is, $(u,x, v)$ is an edge from a vertex $u$ to a vertex $v$ if and only if $(v,x^{-1},u)$ is an edge from $v$ to $u$. A \textit{birooted inverse word graph} is an inverse word graph $\Gamma$ with vertices $\alpha,\beta \in V(\Gamma)$ identified as the start and end vertices, respectively. The language $L[A]$ of a birooted inverse word graph $A=(\alpha, \Gamma, \beta)$ is the set of words that label a path from $\alpha$ to $\beta$ in $\Gamma$.  In a birooted inverse word graph over a presentation $\langle X|R\rangle$, for each relation $(r,s)\in R$ and a vertex $v$, if $r$ and $s$ can be read at $v$, then there is a \textit{region} with boundary given by the pair of paths labeled by $r$ and $s$ starting from $v$. Every region is simply connected, and so is homeomorphic to the open disk.In a birooted inverse word graph, for each relation $(r,s)\in R$ and a vertex $v$, if we can find a segment labeled by one side of the relation $r$ starting from the vertex $v$, but we do not find a segment labeled by the other side $s$ of the relation starting from $v$, then the segment labeled by $r$ is called an \textit{unsaturated segment}. 

 J. B. Stephen \cite{ST} introduced the notion of \textit{Sch\"{u}tzenberger graphs} to solve the word problem for inverse semigroups. If $M=Inv\langle X|R\rangle$ is an inverse semigroup then we may consider the corresponding Cayley graph $\Gamma(M,X)$. The vertices of this graph are labeled by the elements of $M$ and there exists a directed edge labeled by $x\in X\cup X^{-1}$ from the vertex labeled by $m_1$ to the vertex labeled by $m_2$ if $m_2 = m_1x$.   The Cayley graph $\Gamma(M,X)$ is not necessarily strongly connected unless $M$ happens to be a group because it may happen that when there is an edge labeled by $x$ from $m_1$ to $m_2$, there is no edge labeled by $x^{-1}$ from $m_2$ to $m_1$ (so, $m_2 = m_1x$, but $m_1 \neq m_2x^{-1}$). The strongly connected components of $\Gamma(M,X)$ are called the \textit{Sch\"{u}tzenberger graphs} of $M$. For any word $u\in (X\cup X^{-1})^*$, the strongly connected component of $\Gamma(M,X)$ that contains the vertex corresponding to  $u$ is the \textit{Sch\"{u}tzenberger graph of $u$} and it is denoted by $S\Gamma(M,X,u)$. In \cite{ST} it is shown that the vertices of $S\Gamma(M,X,u)$ are precisely those vertices that are labeled by the elements of the $\mathscr{R}$-class of $u$, i.e., $R_u = \{m\in M \mid mm^{-1} = uu^{-1}\}$.

 For any word $u\in (X\cup X^{-1})^*$, it is useful to consider the \textit{Sch\"{u}tzenberger automaton} $(uu^{-1},S\Gamma(M,X,u),u)$ with initial vertex $uu^{-1}\in M$, terminal vertex $u\in M$ and with the Sch\"{u}tzenberger graph of $u$ as the underlying graph. The language accepted by this automaton is a subset of $(X\cup X^{-1})^*$ and will be denoted as $L(u)$.
 \[ L(u) = \{w\in (X\cup X^{-1})^* \mid w \mbox{ labels a path  from } uu^{-1} \mbox{ to } u \mbox{ in } S\Gamma(M,X,u)\}. \]
Here, $u$ and $w$ may be regarded both as elements of $(X\cup X^{-1})^*$ and as elements of $M$. Thus, $L(u)$ may be regarded as a subset of $(X\cup X^{-1})^*$ or as a subset of $M$.

The following result of Stephen \cite{ST} plays a key role in solving the word problem for inverse semigroups.
\begin{theorem}\label{Stephen's thm}
Let $M=Inv\langle X|R\rangle$ and let $u,v\in (X\cup X^{-1})^*$. 
\begin{enumerate}
\item $L(u)= \{w \mid w \geq u$ in the natural partial order on $M\}$.
\item The following are equivalent:
\begin{enumerate}
\item[(i)] $u = v$ in $M$.
\item[(ii)] $L(u)=L(v)$.
\item[(iii)] $u\in L(v)$ and $v\in L(u)$.
\item[(iv)] $(uu^{-1},S\Gamma(M,X,u), u)$ and $(vv^{-1},S\Gamma(M,X,v),v)$ are isomorphic as automata.
\end{enumerate}
\end{enumerate}

\end{theorem}

 We briefly describe the iterative procedure described by Stephen \cite{ST} for building a Sch\"{u}tzenberger graph. Let $Inv\langle X|R\rangle$  be a presentation of an inverse monoid.

   Given a word $u=a_1a_2...a_n\in (X\cup X^{-1})^*$, the \textit{linear graph} of $u$  is the birooted inverse word graph $(\alpha_u,\Gamma_u,\beta_u)$ consisting of the set of vertices

   \begin{center}

   $V((\alpha_u,\Gamma_u,\beta_u))=\{\alpha_u,\beta_u,\gamma_1,...,\gamma_{n-1}\}$

   \end{center}

   and edges

   \begin{center}

   $(\alpha _u,a_1, \gamma _1),(\gamma _1,a_2,\gamma _2),..., (\gamma _{n-2},a_{n-1},\gamma _{n-1}),(\gamma _{n-1},a_n,\beta _u)$,

   \end{center}

    together with the corresponding inverse edges.

    Let $(\alpha , \Gamma ,\beta )$ be a birooted inverse word graph over $X\cup X^{-1}$. The following operations may be used to obtain a new birooted inverse word graph $(\alpha ',\Gamma ',\beta ')$:

    $\bullet$ \textbf{Determination} or \textbf{folding:} Let $(\alpha,\Gamma,\beta)$ be a birooted inverse word graph with vertices $v,v_1,v_2$, with $v_1\neq v_2$, and edges $(v,x,v_1)$ and $(v,x,v_2)$ for some $x\in X\cup X^{-1}$.

    Then we obtain a new birooted inverse word graph $(\alpha',\Gamma',\beta')$ via taking the quotient of $(\alpha,\Gamma,\beta)$ by the equivalence relation that identifies the vertices $v_1$ and $v_2$ and the two edges. In other words, edges with the same label coming out of a vertex are folded together to become one edge.

    $\bullet$ \textbf{Elementary $\mathscr{P}$-expansion:} Let $r=s$ be a relation in $R$ and $r$ can be read from $v_1$ to $v_2$ in $\Gamma$, but $s$ cannot be read from $v_1$ to $v_2$ in $\Gamma$. Then we define $(\alpha',\Gamma',\beta')$ to be the quotient of $\Gamma \cup (\alpha _s,\Gamma_s,\beta_s)$ by the equivalence relation that identifies vertices $v_1$ and $\alpha_s$ and vertices $v_2$ and $\beta_s$. In other words, we  ``sew" on a linear graph for $s$ from $v_1$ to $v_2$ to complete the other half of the relation $r=s$.

    An inverse word graph is \textit{deterministic} if no folding can be performed and is \textit{closed} if it is deterministic and no elementary expansion can be performed over a presentation $\langle X|R\rangle$. Note that given a finite inverse word graph it is always possible to produce a determinized form of the graph because determination reduces the number of vertices. So, the process of determination must stop after finitely many steps. We also observe that the process of folding is confluent \cite{ST} .

    If $(\alpha_1,\Gamma_1, \beta_1)$ is obtained from $(\alpha,\Gamma,\beta)$ by an elementary $\mathscr{P}$-expansion, and $(\alpha_2,\Gamma_2,\beta_2)$ is the determinized  form of $(\alpha_1,\Gamma_1,\beta_1)$, then we write $(\alpha,\Gamma,\beta)$\\$\Rightarrow (\alpha_2,\Gamma_2,\beta_2)$ and say that $(\alpha_2,\Gamma_2,\beta_2)$ is obtained from $(\alpha, \Gamma,\beta)$ by a  \textit{$\mathscr{P}$-expansion}. The reflexive and transitive closure of $\Rightarrow$ is denoted by $\Rightarrow ^*$.

    For $u\in (X\cup X^{-1})^*$, an \textit{approximate graph} of $(uu^{-1}, S\Gamma(u), u)$ is a birooted inverse word graph $A=(\alpha,\Gamma,\beta)$ such that $u\in L[A]$ and $y\geq u$ holds in $M$ for all $y\in L[A]$. Stephen showed in \cite{ST} that the linear automaton of $u$ is an approximate graph of $(uu^{-1}, S\Gamma(u), u)$. He also proved the following:

    \begin{theorem}\label{closure}
    Let $u\in (X\cup X^{-1})^*$ and let $(\alpha,\Gamma,\beta)$ be an approximate graph of $(uu^{-1},S\Gamma(u), u)$. If $(\alpha,\Gamma,\beta)\Rightarrow^*(\alpha',\Gamma',\beta')$ and $(\alpha',\Gamma',\beta')$ is closed, then $(\alpha',\Gamma',\beta')$ is the Sch\"{u}tzenberger automaton for $u$.
    \end{theorem}

    In \cite{ST}, Stephen showed that the class of all birooted inverse words graphs over $X\cup X^{-1}$ is a co-complete category  and that the directed system of all finite $\mathscr{P}$-expansions of a linear graph of $u$ has a direct limit. Since the directed system includes all possible $\mathscr{P}$-expansions, this limit must be closed. Therefore, by \ref{closure}, the Sch\"{u}tzenberger graph of $u$ is the direct limit.

    \textbf{Full $\mathscr{P}$- expansion (a generalization of the concept of $\mathscr{P}$-\\ expansion):} A full $\mathscr{P}$-expansion of a birooted inverse word graph $(\alpha,\Gamma,\beta)$ is obtained in the following way:

    $\bullet$ Form the graph $(\alpha',\Gamma',\beta')$, which is obtained from $(\alpha,\Gamma,\beta)$ by performing all possible elementary $\mathscr{P}$-expansions of $(\alpha,\Gamma,\beta)$, relative to $(\alpha,\Gamma,\beta)$. We emphasize  that an elementary $\mathscr{P}$-expansion may introduce a path labeled by one side of relation in $R$, but we do not perform an elementary $\mathscr{P}$-expansion that could not be done to $(\alpha,\Gamma,\beta)$ when we do a full $\mathscr{P}$-expansion.

    $\bullet$ Find the determinized form $(\alpha_1,\Gamma_1,\beta_1)$, of $(\alpha',\Gamma',\beta')$.

    The birooted inverse word graph $(\alpha_1,\Gamma_1,\beta_1)$ is called the full $\mathscr{P}$-expansion of $(\alpha,\Gamma,\beta)$. We denote this relationship by $(\alpha,\Gamma,\beta)\Rightarrow_f (\alpha_1,\Gamma_1,\beta_1)$.  If $(\alpha_n,\Gamma_n,\beta_n)$ is obtained from $(\alpha,\Gamma ,\beta)$ by a sequence of full $\mathscr{P}$-expansions then we denote this by   $(\alpha,\Gamma ,\beta)\Rightarrow^*_f(\alpha_n,\Gamma_n,\beta_n)$.

For any word $w\in (X\cup X^{-1})^*$, the sequence of birooted approximate graphs $\{(\alpha_i,\Gamma_i(w),\beta_i):i\in I\}$ converges to the Sch\"{u}tzenberger graph of $w$.  In a finitely presented inverse semigroup $Inv\langle X|R\rangle$, there exists a graph morphism $\phi_i:(\alpha_i,\Gamma_i(w),\beta_i)\to (\alpha_{i+1},\Gamma_{i+1}(w),\beta_{i+1})$, for any $w\in (X\cup X^{-1})^*$, and $i\in I$. If  $Inv\langle X|R\rangle$ happens to be a finitely presented Adian inverse semigroup, and $w\in X^+$, then it has been proved in Proposition 3 of \cite{AD} that the birooted graph $(\alpha_i,\Gamma_i(w),\beta_i)$ embeds in $ (\alpha_{i+1},\Gamma_{i+1}(w),\beta_{i+1})$, for all $i\in I$.  Those regions which appear in $(\alpha_i,\Gamma_i(w),\beta_i)$ as a consequence of application of full $\mathscr{P}$- expansion on $ (\alpha_{i-1},\Gamma_{i-1}(w),\beta_{i-1})$, are called the \textit{$i$-th generation regions}, for all $i\in I$. 

\section{Subgraphs generated by an $R$-word of a Sch\"{u}tzenberger graph}

The following lemma is proved in \cite{AD} and it implies that we only use elementary $ \mathscr{P}$-expansions, and no foldings in the construction of $S\Gamma(w)$, for some $w\in X^+$.

\begin{lemma}\label{1} Let $M=Inv\langle X,R\rangle$ be an Adian inverse semigroup and $w\in X^+$. Then no two edges fold together in Stephen's process of constructing approximations of the Sch\"{u}tzenberger graph of $w$. 
\end{lemma}

For a finitely presented Adian inverse semigroup $Inv\langle X|R\rangle$ and a positive word $w\in X^+$ that contains an $R$-word $r$ as its subword, a birooted inverse word subgraph of $S\Gamma(w)$ can be generated by $r$ by using an iterative procedure similar to the Stephen's full $\mathscr{P}$-expansion. The word $w$ can contain the $R$-word $r$ only a finite number of times as its subword. Therefore, we can label each occurrence of $r$ by a number $i\in\mathbb{N}$, starting from the initial letter of $w$ and going along $w$ up to the terminal letter of $w$. So, $r_w(i)$ denotes the $i$-th occurrence of $r$ in $w$, and $\Delta(r_w(i))$ denotes the birooted inverse word subgraph of $S\Gamma(w)$ generated by the $i$-th occurrence of $r$ in $w$. We construct $\Delta(r_w(i))$ by applying elementary $\mathscr{P}$-expansions successively in the following manner. 

We construct the linear automaton of $w$, and denote it by $(\alpha_0,\Delta_0(r_w(i)), \beta_0)$, where the underlying graph is denoted by $\Delta_0(r_w(i))$, and $\alpha_0$, and $\beta_0$ are the initial and terminal vertices of the underlying linear graph.


 At the first step, we only apply the elementary $\mathscr{P}$-expansion on the $i$-th segment labeled by $r$ by sewing on a segment labeled by $s$ from the initial vertex of the segment $r$ to the terminal vertex of the segment $r$, for some $(r,s)\in R$. We denote the resulting graph by $(\alpha_1,\Delta_1(r_w(i)), \beta_1)$, where the underlying graph is denoted by $\Delta_1(r_w(i))$, and $\alpha_1$, and $\beta_1$ are the initial and terminal vertices of the underlying graph. This creates a first generation region whose one side is labeled by $r$, and it lies on the linear automaton of $w$. In $(\alpha_1,\Delta_1(r_w(i)), \beta_1)$, if we cannot find any unsaturated segment labeled by some $R$-word that either starts from or terminates at an interior vertex of the segment $s$ that was sewn on at the first step, then we cannot continue further.  In this case, the graph $\Delta_1(r_w(i))$ is the subgraph of $S\Gamma(w)$ generated by the $i$-th occurrence of the $R$-word $r$.  Otherwise, there will be some unsaturated segments labeled by some $R$-words $r_j$, for $1\leq j\leq n_1$, that either start from or terminate at an interior vertex of the segment $s$ that was sewn on at the first step. These new unsaturated segments share an edge with the segment $s$ that was sewn on at the first step. 

 At the second step, we only apply elementary $\mathscr{P}$-expansion on all these new unsaturated segments labeled by  $r_j$'s by sewing on segments labeled by $s_j$'s from the initial vertices to the terminal vertices of the corresponding segments $r_i$'s, where $\{(r_j,s_j), 1\leq j\leq n_1\}\subseteq R$. This step creates the second generation regions of $\Delta(r_w(i))$. We denote the resulting graph by $(\alpha_2,\Delta_2(r_w(i)), \beta_2)$, where the underlying graph is denoted by $\Delta_2(r_w(i))$, and $\alpha_2$, and $\beta_2$ are the initial and terminal vertices of the underlying graph.  All these second generation regions share an edge with the first generation region created at the first step. If we cannot find any unsaturated segment labeled by an $R$-word that either starts from an interior vertex of the segment $s_j$ or terminates an in interior vertex of the segment $s_j$, for some $j\in\{1,2,...,n_1\}$, where the segments $s_j$'s were sewn on at the second step, then we cannot proceed further. In this case, the graph $\Delta_2(r_w(i))$ is the subgraph of $S\Gamma(w)$ generated by the $i$-th occurrence of the $R$-word $r$. Otherwise, there will be some unsaturated segments labeled by some $R$-words $u_j$'s, for $1\leq j\leq n_2$, that either start from an interior vertex of the segment $s_j$ or terminate at an interior vertex of $s_j$, for some $j\in\{1,2,...,n_1\}$, where the segments $s_j$'s were sewn on at the second step. These new unsaturated segments $u_j$'s share an edge with some of the segments $s_j$'s that were sewn on at the second step. 

 At the third step, we only apply elementary $\mathscr{P}$-expansion on all these new unsaturated segments $u_j$'s by sewing on segments labeled by $v_j$'s from the initial vertices to the terminal vertices of the corresponding segments $u_j$'s, where $\{(u_j,v_j), 1\leq j\leq n_2\}\subseteq R$. We denote the resulting graph by $(\alpha_3,\Delta_3(r_w(i)), \beta_3)$, where the underlying graph is denoted by $\Delta_3(r_w(i))$, and $\alpha_3$, and $\beta_3$ are the initial and terminal vertices of the underlying graph. This step forms all the third generation regions of $\Delta(r_w(i))$. Note that, all the third generation regions share an edge with a second generation region. If we cannot find any unsaturated segment labeled by an $R$-word that either starts from or terminates an in interior vertex of the segment $v_j$, for some $j\in\{1,2,...,n_2\}$, where the segments $v_j$'s were sewn on at the third step, then we cannot proceed further. In this case, the graph $\Delta_3(r_w(i))$ is the subgraph of $S\Gamma(w)$ generated by the $i$-th occurrence of the $R$-word $r$. Otherwise,  there will be some unsaturated segments labeled by some $R$-words that either start from or terminate at an in interior vertex of the segment $v_j$, for some $j\in\{1,2,...,n_2\}$, where the segments $v_j$'s were sewn on at the third step. So, we continue the process of successive applications of elementary $\mathscr{P}$-expansions in this manner,  and we obtain a sequence of approximate graphs $\{(\alpha_n,\Delta_n(r_w(i)),\beta_n):n\in\mathbb{N}\}$  that converges to the subgraph $\Delta(r_w(i))$ of $S\Gamma(w)$. If the sequence $\{(\alpha_n,\Delta_n(r_w(i)),\beta_n):n\in\mathbb{N}\}$ stabilizes after a finite number of terms, then $\Delta(r_w(i))$ is finite. Otherwise, $\Delta(r_w(i))$ is infinite. 

The construction of $\Delta(r_w(i))$ can be summarized in the following way: 

\begin{enumerate}

\item We construct the linear automaton of $w$, and denote it by $(\alpha_0,\Delta_0(r_w$ $(i)), \beta_0)$, where the underlying graph is denoted by $\Delta_0(r_w(i))$, and $\alpha_0$, and $\beta_0$ are the initial and terminal vertices of the underlying graph.

\item We only apply the elementary $\mathscr{P}$-expansion on the $i$-th segment labeled by $r$. We denote the resulting graph by $(\alpha_1,\Delta_1(r_w(i)), \beta_1)$, where the underlying graph is denoted by $\Delta_1(r_w(i))$, and $\alpha_1$, and $\beta_1$ are the initial and terminal vertices of the underlying graph.
. 

\item For $n\geq 1$, to obtain $(\alpha_{n+1},\Delta_{n+1}(r_w(i)),\beta_{n+1})$ from $(\alpha_n,\Delta_n(r_w(i)),$ $\beta_n)$, we only apply elementary $\mathscr{P}$-expansion on those unsaturated segments labeled by some $R$-words of $\Delta_n(r_w(i))$ that either start from or terminate at an interior vertex of the segments labeled by some $R$-words and were sewn on  at the $n$-th iterative step.

\end{enumerate}
Consequently, we obtain a sequence of approximate graphs $\{(\alpha_n,\Delta_n(r_w$ $(i)), \beta_n):n\in\mathbb{N}\}$  that converges to the subgraph $\Delta(r_w(i))$ of $S\Gamma(w)$. 

\begin{Ex}
For instance, we consider the inverse semigroup $Inv\langle a,b|ab=ba\rangle$ and the word $w\equiv a^2b^2a^2b^2\in\{a,b\}^+$. The word $w$ contains both the $R$-words $ab$ and $ba$ as its subwords. The $R$-word $ab$ occurs twice and the $R$-word $ba$ occurs once in $w$. We construct the subgraph of $S\Gamma(w)$ generated by the first occurrence of the $R$-word $ab$. 

First, we construct the linear automaton of $w$, $(\alpha_0, \Delta_0(ab_w(1)),\beta_0)$ (see figure \ref{5a}). We apply the elementary $\mathscr{P}$-expansion only at the first subsegment labeled by $ab$ of $w$ by sewing on a path labeled by $ba$ from the initial vertex of $ab$ to the terminal vertex of $ab$, and we denote the resulting graph by $(\alpha_1, \Delta_1(ab_w(1)),\beta_1)$ (see figure \ref{5b}). In $\Delta_1(ab_w(1))$, we can find two unsaturated segments labeled by $ab$, one of which terminates at an interior vertex of the segment $ba$, and the other one starts from an interior vertex of the segment $ba$, where the segment labeled by $ba$ was sewn at the previous step. We apply elementary $\mathscr{P}$-expansion only on both of these unsaturated segments by sewing on segments labeled by $ba$ from the initial vertices to the terminal vertices of these unsaturated segments. We denote the resulting graph by  $(\alpha_2, \Delta_2(ab_w(1)),\beta_2)$ (see figure \ref{5c}). In  $\Delta_2(ab_w(1)$, we can find only one unsaturated segment labeled by $ab$ that starts from an interior vertex of a segment labeled by $ba$ that was sewn on at the previous step. So, we apply elementary $\mathscr{P}$-expansion only on this unsaturated segment labeled by sewing on a path labeled by $ba$ from the initial vertex to the terminal vertex of the segment labeled by $ab$. We denote the resulting graph by $(\alpha_3, \Delta_3(ab_w(1)),\beta_3)$ (see figure \ref{5d}). In $\Delta_3(ab_w(1))$, we cannot find any unsaturated segments labeled by an $R$-word that either starts from or terminates at an interior vertex of the segment $ba$ that was sewn on at the previous step. Therefore, we cannot continue further. Hence $\Delta_3(ab_w(1))$ is the subgraph of $S\Gamma(w)$ generated by the first occurrence of the $R$-word $ab$ in $w$, which is denoted by $\Delta(ab_w(1))$.
 \begin{figure}[h!]
\centering
\includegraphics[trim = 0mm 0mm 0mm 0mm, clip,width=2.8in]{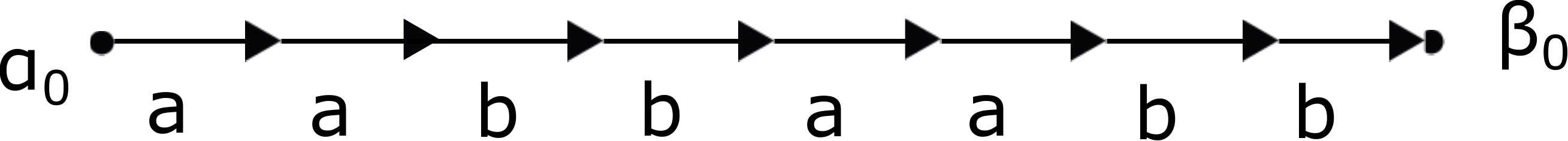}
\caption{$(\alpha_0, \Delta_0(ab_w(1)),\beta_0)$}
\label{5a}
\end{figure}

 \begin{figure}[h!]
\centering
\includegraphics[trim = 0mm 0mm 0mm 0mm, clip,width=2.8in]{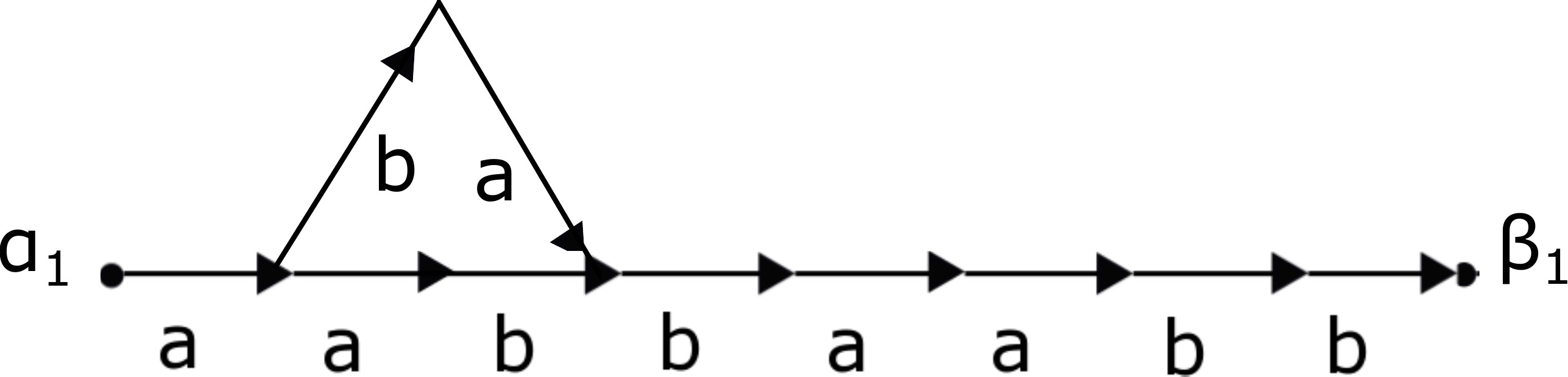}
\caption{$(\alpha_1, \Delta_1(ab_w(1)),\beta_1)$}
\label{5b}
\end{figure}

 \begin{figure}[h!]
\centering
\includegraphics[trim = 0mm 0mm 0mm 0mm, clip,width=2.8in]{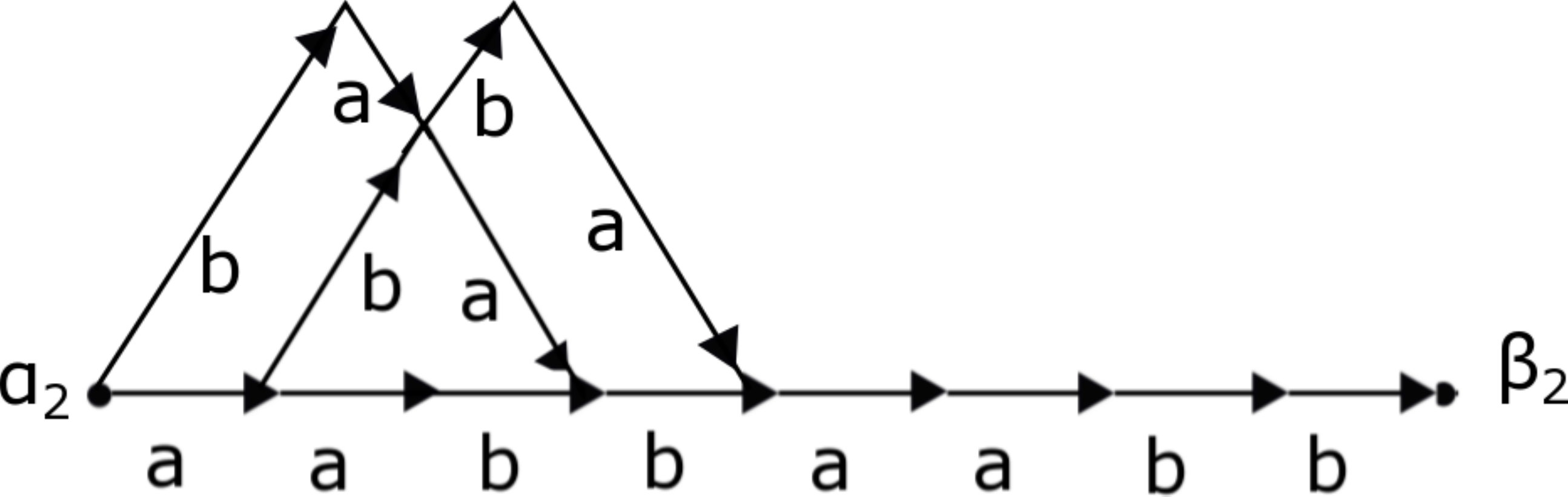}
\caption{$(\alpha_2, \Delta_2(ab_w(1)),\beta_2)$}
\label{5c}
\end{figure}

 \begin{figure}[h!]
\centering
\includegraphics[trim = 0mm 0mm 0mm 0mm, clip,width=2.8in]{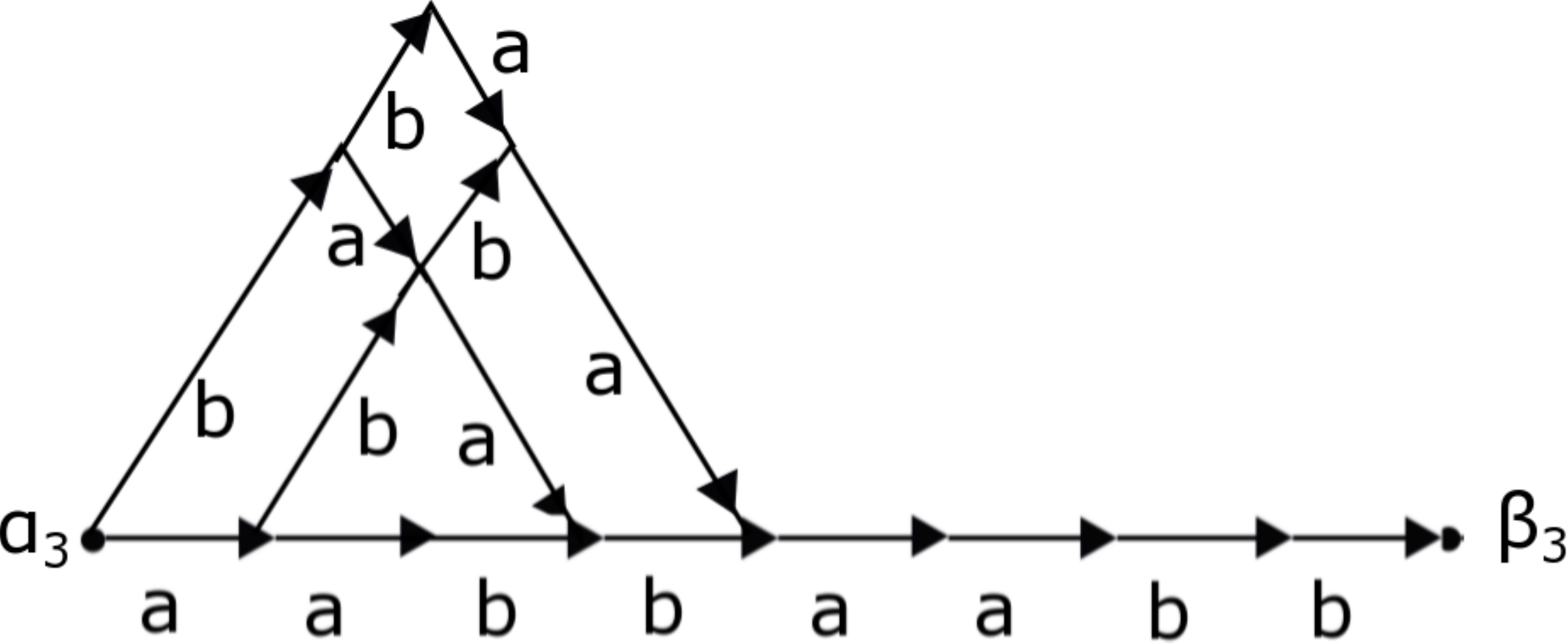}
\caption{$(\alpha_3, \Delta_3(ab_w(1)),\beta_3)$}
\label{5d}
\end{figure}

Similarly, we can construct $\Delta(ab_w(2))$  (see figure \ref{6a}) and $\Delta(ba_w(1))$  (see figure \ref{7a}). 

 \begin{figure}[h!]
\centering
\includegraphics[trim = 0mm 0mm 0mm 0mm, clip,width=2.8in]{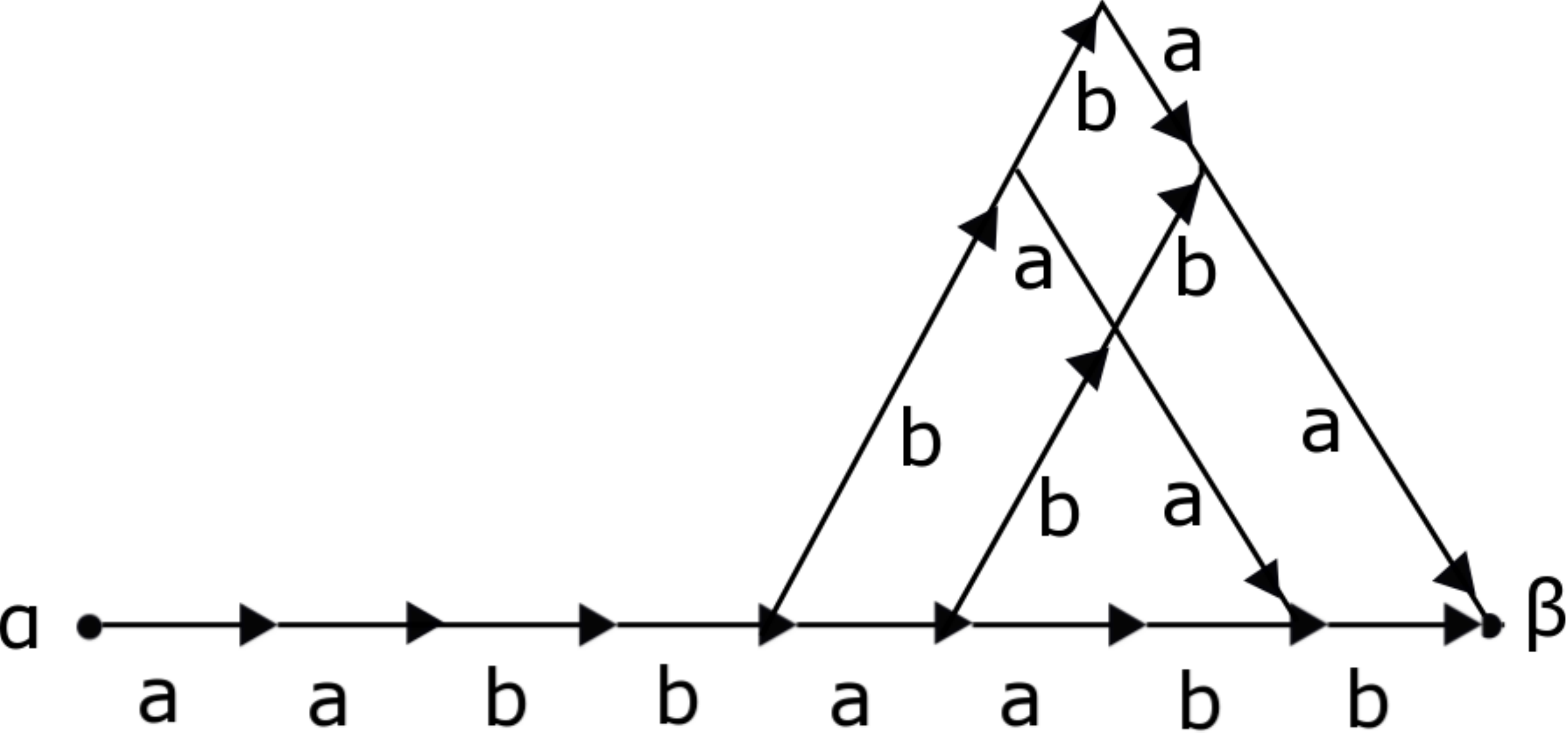}
\caption{$\Delta(ab_w(2))$}
\label{6a}
\end{figure}

 \begin{figure}[h!]
\centering
\includegraphics[trim = 0mm 0mm 0mm 0mm, clip,width=2.8in]{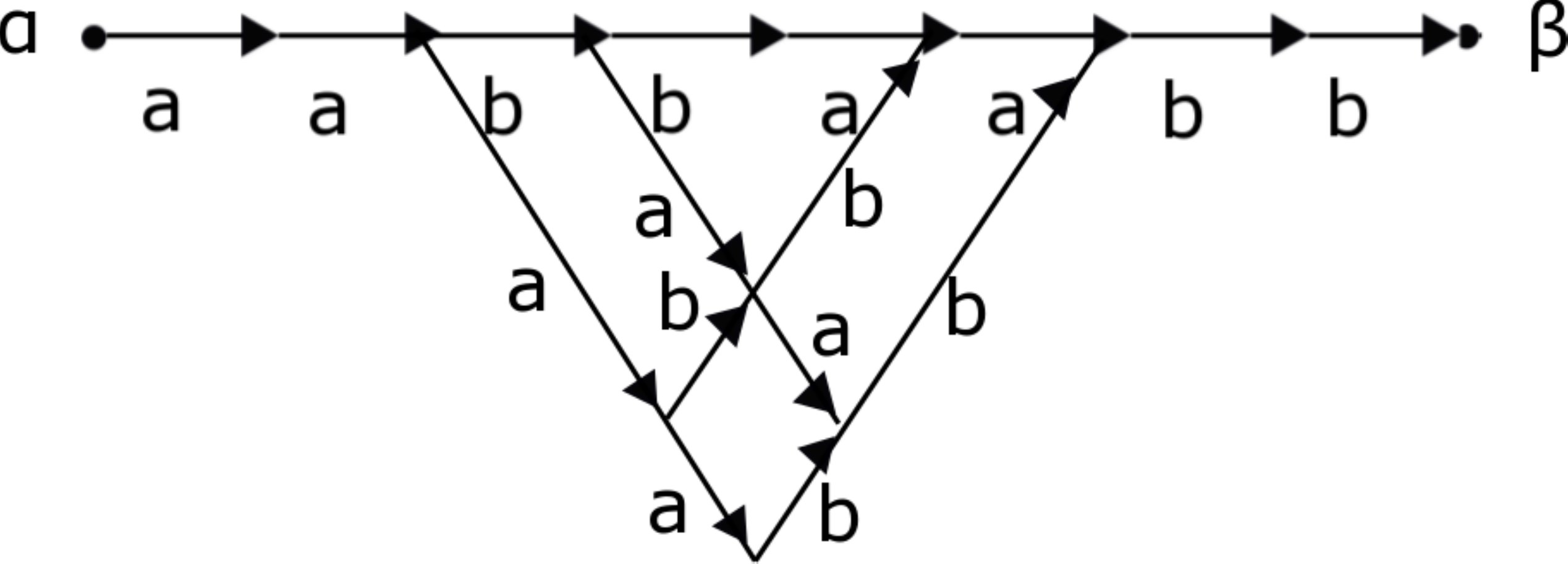}
\caption{$\Delta(ba_w(1))$}
\label{7a}
\end{figure}

The complete graph  of $S\Gamma(w)$ can be constructed by five successive applications of Stephen's full $\mathscr{P}$-expansion (see figure \ref{8a}).

 \begin{figure}[h!]
\centering
\includegraphics[trim = 0mm 0mm 0mm 0mm, clip,width=2.8in]{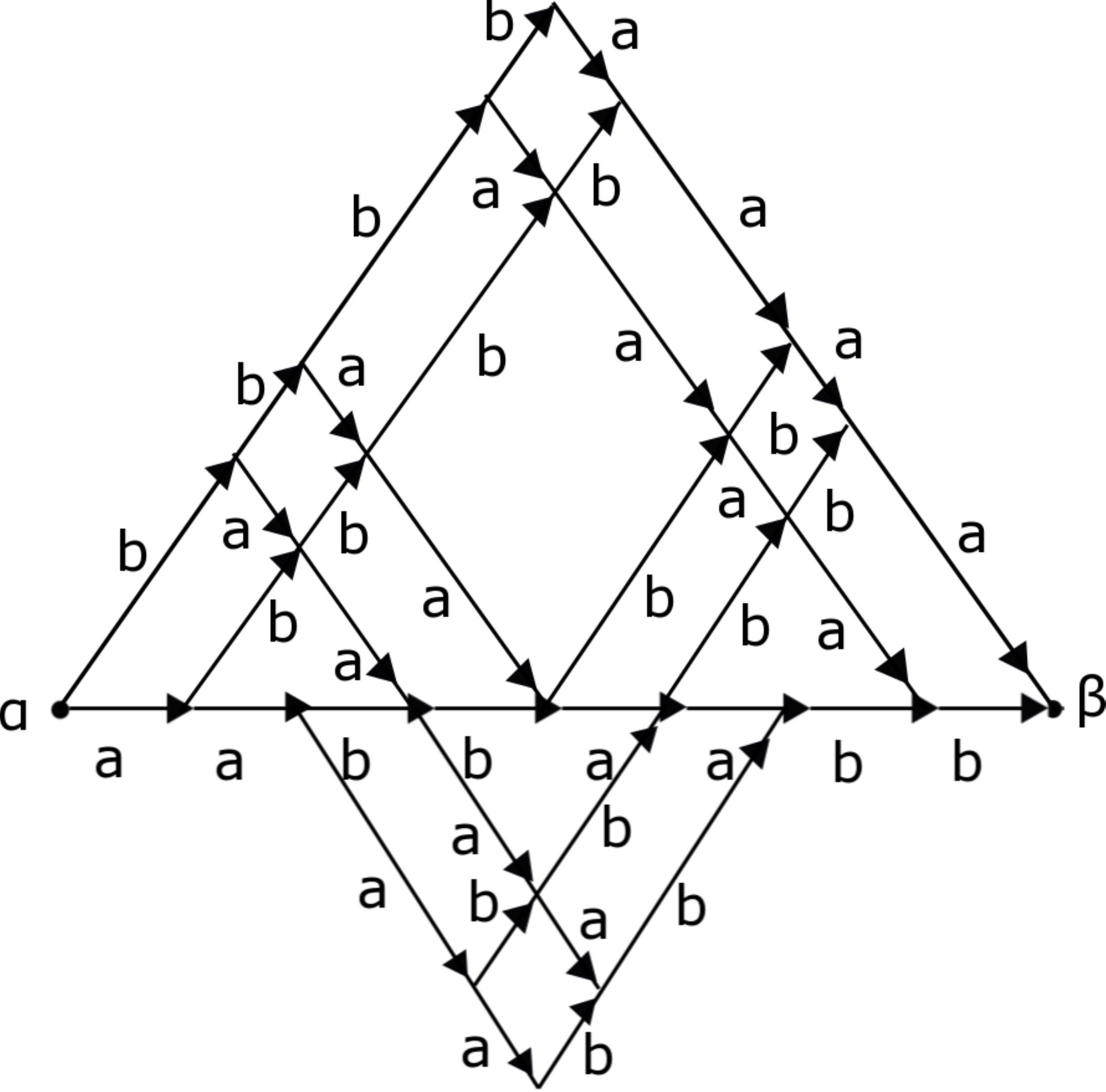}
\caption{$S\Gamma(w)$}
\label{8a}
\end{figure}

\end{Ex}

The next Lemma and Proposition show that for a finitely presented Adian inverse semigroup $Inv\langle X|R\rangle$, and a positive word $w\in X^+$ that contains an $R$-word $r$ as its subword, $\Delta(r_w(i))$ inherits  most of the structural properties of $S\Gamma(w)$, for all $i\in I$. 

The set of all edges of a graph whose tail vertex lies at a vertex $\gamma$ is denoted by $Star^o(\gamma)$, and all edges whose tip lies at a vertex $\gamma$ is denoted by $Star^i(\gamma)$. 

\begin{lemma}\label{2} Let $M=\langle X|R\rangle$ be a finitely presented Adian inverse semigroup, and $w\in X^+$ be a positive word that contains an $R$-word $r$ as its subword. The birooted inverse word graph $(\alpha_n,\Delta_n(r_w(i)),\beta_n)$ is deterministic for all $n\in\mathbb{N}$ and $i\in I$. 
\end{lemma}

\begin{proof} We use induction on the number of iterative steps involved in the construction of  $(\alpha_n,\Delta_n(r_w(i)),\beta_n)$. 

As the base case for induction, we begin with the linear automaton of $w$, $(\alpha_0,\Delta_0(r_w(i)),\beta_0)$. Because $w\in X^+$ is a positive word, the linear automaton has no two oppositely oriented consecutive  edges. So, the linear automaton of $w$ is deterministic. 

Now we suppose that $(\alpha_n,\Delta_n(r_w(i)),\beta_n)$, for some $n\geq 1$, is deterministic. Suppose that $(u,v)\in R$ is a relation such that $v$ labels a segment from a vertex $\gamma_1$ to a vertex $\gamma_2$ of $(\alpha_n,\Delta_n(r_w(i)),\beta_n)$, but $u$ does not label a segment from $\gamma_1$ to $\gamma_2$, where $\gamma_1$ is an interior vertex of a segment  labeled by an $R$-word that was sewed on at the $n$-th iterative step of the construction of $\Delta(r_w(i))$. Let $a_1,a_2,...,a_n$ be the labels of the edges in $Star^o(\gamma_1)$. Since $(\alpha_n,\Delta_n(r_w(i)),\beta_n)$ is deterministic, $a_1,a_2,...,a_n$ are the vertices of a subtree $T$ of $LG\langle X|R\rangle$. We assume that the prefix letter of $u$ is $a_{n+1}$ and the prefix letter of $v$ is $a_n$. Then $a_{n}$ and $a_{n+1}$ cannot be the same letters because there exists an edge between $a_{n}$ and $a_{n+1}$, corresponding to the relation $(u,v)$ in $LG\langle X|R\rangle$. So, if $a_n$ and $a_{n+1}$ are the same letters then there will be a cycle in $LG\langle X|R\rangle$, which is a contradiction. The letter $a_{n+1}$ cannot be the same as $a_j$, for some $ j\in\{1,2,...,{n-1}\}$. Otherwise, there exists a path between $a_n$ and $a_{n+1}$ in $T$. Also, there exists an edge in $LG\langle X|R\rangle$ between $a_n$ and $a_{n+1}$, corresponding to the relation $(u,v)$, which is not included in $T$. This implies that there is a cycle in $LG\langle X|R\rangle$, which is a contradiction. Hence, sewing on a segment labeled by $u$ from $\gamma_1$ to $\gamma_2$ will not produce two edges with the same label in $Star^o(\gamma_1)$. A dual argument, using $RG\langle X|R\rangle$, holds for incoming edges at an interior vertex of a segment labeled by an $R$-word, and sewed on at the $n$-th iterative step of the construction of $\Delta(r_w(i))$.

\end{proof} 

For an Adian inverse semigroup $Inv\langle X|R\rangle$ and a positive word $w\in X^+$ that contains an $R$-word $r$ as its subword, there exist graph homomorphisms $\theta_n:(\alpha_n,\Delta_n(r_w(i)),\beta_n)\to (\alpha_n,\Gamma_n(w),\beta_n)$, such that $\theta_n(\alpha_n)=\alpha_n$ and $\theta_n(\beta_n)=\beta_n$, for all $n\in \mathbb{N}$, and $i\in I$. Also, there exist graph homomorphisms $\phi_n:(\alpha_n,\Delta_n(r_w(i)),\beta_n)\to (\alpha_{n+1},\Delta_{n+1}(r_w(i)),\beta_{n+1})$ such that $\phi_n(\alpha_n)=\alpha_{n+1}$ and $\phi_n(\beta_n)=\beta_{n+1}$, for all $n\in \mathbb{N}$, and $i\in I$. 

\begin{proposition}\label{3} 
Let $M=Inv\langle x|R\rangle$ be an Adian inverse semigroup and $w\in X^+$ be a positive word that contains an $R$-word $r$ as its subword. Then, for all $i\in I$:

$(i).$ $\theta_n:(\alpha_n,\Delta_n(r_w(i)),\beta_n)\to (\alpha_n,\Gamma_n(w),\beta_n)$ is an embedding for all $n\in \mathbb{N}$. 

$(ii).$ $\phi_n:(\alpha_n,\Delta_n(r_w(i)),\beta_n)\to (\alpha_{n+1},\Delta_{n+1}(r_w(i)),\beta_{n+1})$ is an embedding for all $n\in \mathbb{N}$. 

$(iii).$ $\Delta(r_w(i))$ has exactly one source vertex $\alpha$ and one sink vertex $\beta$, where $(\alpha,\Delta(r_w(i)),\beta)$ is the subgraph of $S\Gamma(w)$ generated by the $i$-th occurrence of $r$ in $w$.

$(iv)$. For every vertex $\gamma\neq \alpha$ in $(\alpha,\Delta(r_w(i)),\beta)$ there exists a positively labeled path in $\Delta(r_w(i))$ from $\alpha$ to $\gamma$. For every vertex $\gamma\neq\beta$ there exists a positively labeled path in $\Delta(r_w(i))$ from $\gamma$ to $\beta$. 

$(v).$ Every positively labeled path in $\Delta(r_w(i))$ can be extended to a positively labeled transversal from $\alpha$ to $\beta$. 

\end{proposition}

\begin{proof}$(i)$. Every vertex of the approximate graph $(\alpha_n,\Delta_n(r_w(i)),\beta_n)$ gets mapped to a vertex of the approximate graph $(\alpha_n,\Gamma_n(w),\beta_n)$ with the same label under $\theta_n$, for all $n\in\mathbb{N}$. It follows from Lemma \ref{1} that the approximate graphs $(\alpha_n,\Gamma_n(w),\beta_n)$ are deterministic, for all $n\in \mathbb{N}$. Therefore, the images of two distinct vertices of  $(\alpha_n,\Delta_n(r_w(i)),\beta_n)$ remain distinct under the map $\theta_n$, for all $n\in\mathbb{N}$. 

$(ii)$. From Lemma \ref{2} the approximate graphs $(\alpha_n,\Delta_n(r_w(i)),\beta_n)$ are deterministic, for all $n\in\mathbb{N}$.  So, no folding occurs in the construction of $(\alpha_{n+1},\Delta_{n+1}(r_w(i)),\beta_{n+1})$ from $(\alpha_n,\Delta_n(r_w(i)),\beta_n)$. The images of two distinct vertices of the approximate graph $(\alpha_n,\Delta_n(r_w(i)),\beta_n)$ remain distinct under the map $\phi_n$, for all $n\in\mathbb{N}$. 

$(iii)$. From $(ii)$, the sequence of approximate graphs $\{(\alpha_n,\Delta_n(r_w(i)),\beta_n)$: $n\in \mathbb{N}\}$ converges to $(\alpha,\Delta(r_w(i)),\beta)$ such that $\phi_n(\alpha_n)=\alpha_{n+1}=\alpha$, and $\phi_n(\beta_n)=\beta_{n+1}=\beta$ for all $n\in \mathbb{N}$. We use induction to show that $\alpha$ and $\beta$ are the unique (distinct) source and sink vertices of $(\alpha,\Delta_n(r_w(i)),\beta)$ for all $n\in\mathbb{N}$. 

Notice that $\alpha$ and $\beta$ are the unique (distinct) source and skin vertices of the linear automaton $(\alpha,\Delta_0(r_w(i)),\beta)$.  As our inductive hypothesis, we assume that $\alpha$ and $\beta$ are the unique (distinct) source and sink vertices of $(\alpha, \Delta_n(r_w(i)),\beta)$, for some $n\geq 1$. To obtain $(\alpha, \Delta_{n+1}(r_w(i)),\beta)$ from $(\alpha, \Delta_{n}(r_w(i)),\beta)$, we sew on segments labeled by the other sides of those relations whose one side label an unsaturated segment in the approximate graph $(\alpha,\Delta_n(r_w(i)),\beta)$ and do not perform any folding. So, $\alpha$ and $\beta$ remain unique (distinct) source and sink vertices of $(\alpha, \Delta_{n+1}(r_w(i)),\beta)$. Hence by induction, $\alpha$ and $\beta$ are the unique (distinct) source and sink vertices of $\Delta(r_w(i))$.

$(iv)$. From $(ii)$ and $(iii)$, the sequence of approximate graphs $\{(\alpha,\Delta_n(r_w$ $(i)),\beta): n\in \mathbb{N}\}$ converges to $(\alpha,\Delta(r_w(i)),\beta)$, and each approximate graph $(\alpha,\Delta_n(r_w(i)),\beta)$ can be viewed as a subgraph of $\Delta(r_w(i))$. Therefore, for every vertex $\gamma(\neq \alpha)$ of $\Delta(r_w(i))$ there is a least positive integer $n$ such that $\gamma$ lies in $(\alpha,\Delta_n(r_w(i)),\beta)$, but $\gamma$ does not lie in $(\alpha,\Delta_{n-1}(r_w(i)),\beta)$. We use induction on $n$ to prove the statement $(iv)$.

Let $\gamma(\neq\alpha)$ be a vertex of $\Delta(r_w(i))$. If $\gamma$ lies in $(\alpha,\Delta_0(r_w(i)),$ $\beta)$, then clearly there is a unique positively labeled path from $\alpha$ to $\gamma$ in $(\alpha,\Delta_0(r_w(i)),$ $\beta)$. As our inductive hypothesis, we assume that for every vertex $\delta(\neq\alpha)$ there exists a positively labeled path from $\alpha$ to $\delta$ in $(\alpha,\Delta_n$ $(r_w(i)),\beta)$, for some $n\geq 1$. Now, we assume that $\gamma$ lies in $(\alpha,\Delta_{n+1}(r_w(i)),$ $\beta)$. There are two possibilities. First, If $\gamma$ lies in  $\phi_n(\alpha,\Delta_n(r_w(i)),\beta)$, then by inductive hypothesis there is a positively labeled path from $\alpha$ to $\gamma$. Second, If $\gamma\in V\{(\alpha, \Delta_{n+1}(r_w(i)),\beta)\setminus \phi_n (\alpha, \Delta_n(r_w(i)),\beta)\}$, then $\gamma$ lies on a segment $l$ labeled by an $R$-word, that was sewn on to $(\alpha,\Delta_n(r_w(i)),\beta)$ at the $(n+1)$st iterative step. We denote the initial and terminal vertices of $l$ by $\gamma_0$ and $\gamma_1$. Clearly, $\gamma_0$ and $\gamma_1$ lie in $(\alpha,\Delta_n(r_w(i)),\beta)$. By inductive hypothesis, there is a positively labeled path $p$ from $\alpha$ to $\gamma_0$ in $(\alpha,\Delta_n(r_w(i)),\beta)$. Let $q$ denotes the positively labeled path from $\gamma_0$ to $\gamma$ along $l$. The concatenation of paths $pq$ is a positively labeled path from $\alpha$ to $\gamma$ in $(\alpha,\Delta_{n+1}(r_w(i)),$ $\beta)$.

A dual argument can be used show that for every vertex $\gamma\neq\beta$ there exists a positively labeled path in $\Delta(r_w(i))$ from $\gamma$ to $\beta$.

$(v)$. Let $p$ be a positively labeled path from a vertex $\gamma_1$ to a vertex $\gamma_2$ in $(\alpha, \Delta(r_w(i)),\beta)$. There exists a positively labeled path $p_1$ from $\alpha$ to $\gamma_1$ in $\Delta(r_w(i))$ by $(iv)$. The path $p_1$ is empty if $\alpha$ and $\gamma_1$ represent the same vertex. Similarly, there exists a positively labeled path $p_2$ from $\gamma_2$ to $\beta$ in $\Delta(r_w(i))$ by $(iv)$. The concatenation of paths $p_1pp_2$ is a positively labeled transversal of $\Delta(r_w(i))$ from $\alpha$ to $\beta$. 

\end{proof}

For any semigroup $S=Sg\langle X|R\rangle$, and two words $p,q\in X^+$, $p\to q$ denotes that $q$ is obtained from $p$ by replacing a subword $r$ by $s$, for some $(r,s)\in R$. Any two words $u, v\in X^+$ are equal in the semigroup $S=Sg\langle X|R\rangle$ if and only if there exists a transition sequence from $u$ to $v$,

\begin{centering}
$u\equiv w_0\to w_1\to w_2\to...\to w_n\equiv v$, for some $n\geq 0$,

\end{centering}

The above derivation sequence is called a \textit{regular derivation sequence of length n} for the pair $(u,v))$ over the presentation $Sg\langle X|R\rangle$. 

A \textit{semigroup diagram} or $S-diagram$ over a semigroup presentation $Sg\langle X|R\rangle$ for the pair of positive words $(u,v)$ is a finite planar diagram $D\subseteq \mathbb{R}^2$, that satisfies the following properties:

\begin{itemize}

\item The diagram $D$ is connected and simply connected. 

\item Each edge is directed and labeled by a letter of the alphabet $X$. 

\item Each region of $D$ is labeled by the word $rs^{-1}$ for some defining relation $(r,s)\in R$. 

\item There is a distinguished vertex $\alpha$ on the boundary of $D$ such that the boundary of $D$ starting at $\alpha$ is labeled by $uv^{-1}$. $\alpha$ is a source in $D$ (i.e. there is no edge in $D$ with terminal vertex $\alpha$).

\item There are no interior sources or sinks in $D$. 

\end{itemize} 

In \cite{Remmers}, Remmers proved an analogue of Van Kampen's Lemma for semigroups to address the word problem for semigroups. 

\begin{theorem} Let $S=Sg\langle X|R\rangle$ be a semigroup and $u,v\in X^+$. Then there exists a regular derivation sequence of length $n$ for the pair $(u,v)$ over the presentation $Sg\langle X|R\rangle$ for the pair $(u,v)$ having exactly $n$ regions. 
\end{theorem}

A sub-diagram $\Pi'$ of an $S$-diagram $\Pi$ is called a \textit{simple component} of $\Pi$ if it is a maximal sub-diagram whose boundary is labeled by a simple closed curve.

\begin{theorem}\label{MT} 
Let $M=Inv\langle X|R\rangle$ be a finitely presented Adian inverse semigroup.  The Sch\"{u}tzenberger graph of every positive word  is finite if and only if the subgraphs of $S\Gamma(w)$, for all $w\in X^+$, generated by all the $R$-words are finite.
\end{theorem}

\textit{\textbf{Idea of the proof:}} We assume that the subgraphs of $S\Gamma(t)$, for all $t\in X^+$, generated by all the $R$-words are finite and we let $w$ be an arbitrary positive word, $w\in X^+$. We use induction on the number of segments labeled by the $R$-words of the linear automaton of $w$, $(\alpha_0, \Gamma_0(w), \beta_0)$, to show that $S\Gamma(w)$ is finite. The construction of $S\Gamma(w)$ only involves the applications of elementary $\mathscr{P}$-expansions and no foldings by Lemma \ref{1}. We factorize $w$ as $w\equiv w_1w_2$, such that $w_1$, and $w_2$ contain fewer segments labeled by the $R$-words than $w$. So, by induction hypothesis, $S\Gamma(w_1)$ and $S\Gamma(w_2)$ are finite. To speedup the Stephen's process of successive applications of elementary $\mathscr{P}$-expansions we sew on $S\Gamma(w_1)$ and $S\Gamma(w_2)$ to the corresponding segments of $(\alpha_0, \Gamma_0(w), \beta_0)$, and denote the resulting finite inverse word graph by $S_0$. We observe that in $S_0$ there is only one vertex (labeled by $\gamma$)that is common to $S\Gamma(w_1)$ and $S\Gamma(w_2)$. Since $S\Gamma(w_1)$ and $S\Gamma(w_2)$ are $\mathscr{P}$-complete, therefore we can only find some unsaturated segments passing through $\gamma$. To speedup Stephen's processes of successive applications of $\mathscr{P}$-expansions we sew on the finite subgraphs generated by the $R$-words that label unsaturated segments passing through $\gamma$, to the corresponding transversals of $S_0$, and obtain a finite inverse word graph, $S_1$. Finally, we show that $S_1$ is $\mathscr{P}$-complete. So, $S_1$ is $S\Gamma(w)$.

\begin{proof} The direct statement is obvious. So, we prove the converse statement only.  Let $w$ be an arbitrary word, $w\in X^+$. If $w$ does not contain any $R$-word as its subword, then our claim follows immediately. If $w$ contains only one $R$-word as its subword, such that the $R$-word occurs only once in $w$, then the finiteness of $S\Gamma(w)$ follows from the hypothesis of the theorem.  

We assume that the statement of the theorem is true for any positive word that contains less than $n$ number of (not necessarily distinct) $R$-words as its subword, where $n\in \mathbb{N}$ and $n\geq 2$. We assume that $w$ contains $n$ number of (not necessarily distinct) $R$-words as its subwords. We show that $S\Gamma(w)$ is finite.

 We factorize $w$ as $w\equiv w_1w_2$, such that both $w_1$ and $w_2$ contain less than $n$ $R$-words as their subwords. The Sch\"{u}tzenberger graphs of $w_1$ and $w_2$ are finite by the induction hypothesis. We construct the linear automaton of $w$, $(\alpha_0, \Gamma_0(w),\beta_0)$. We sew on $S\Gamma(w_1)$ and $S\Gamma(w_2)$ to $(\alpha_0, \Gamma_0(w),\beta_0)$ at the segments labeled by $w_1$ and $w_2$, respectively.  We denote the resulting graph by $S_0$ (see Figure \ref{1a}).

 \begin{figure}[h!]
\centering
\includegraphics[trim = 0mm 0mm 0mm 0mm, clip,width=2.8in]{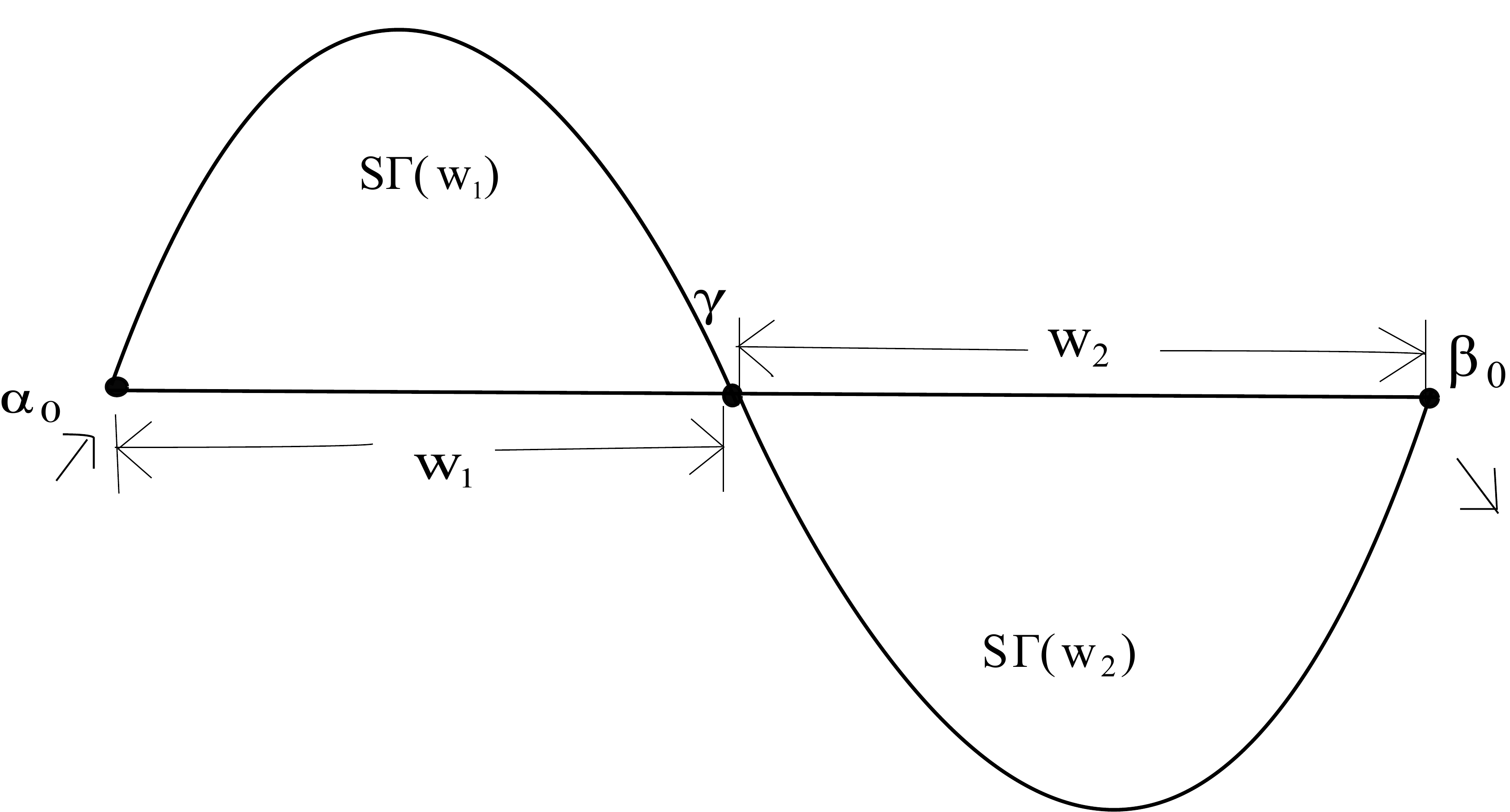}
\caption{$S_0$ with $S\Gamma(w_1)$ and $S\Gamma(w_2)$ attached to the segments $w_1$ and $w_2$}
\label{1a}
\end{figure}

We observe that in $S_0$ there is only one vertex (labeled by $\gamma$) common to $S\Gamma(w_1)$ and $S\Gamma(w_2)$. The vertex $\gamma$ is the terminal vertex of $S\Gamma(w_1)$, and the initial vertex of $S\Gamma(w_2)$. Therefore, every transversal (a positively labeled path from the initial vertex $\alpha_0$ to the terminal vertex $\beta_0$) of $S_0$ passes through $\gamma$ (see Figure \ref{2a}). We know that $S\Gamma (w_1)$ and $S\Gamma (w_2)$ are $\mathscr{P}$-complete. So, in $S_0$ we can only  find some unsaturated segments labeled by some $R$-words starting from a vertex of $S\Gamma(w_1)$, passing though $\gamma$, and terminating at a vertex of $S\Gamma(w_2)$. 
 
By Lemma 2 of \cite{AD}, the Sch\"{u}tzebger graphs $S\Gamma(w_1)$ and $S\Gamma(w_2)$ do not contain any positively labeled cycle (closed path). Therefore, there are only a finite number of transversals in $S_0$, as it is a finite inverse word graph. We consider those transversals of $S_0$ which contain an unsaturated segment labeled by an $R$-word passing through  $\gamma$. We assume that these transversals are labeled by $t_1,t_2,...,t_m$, for some $m\in \mathbb{N}$. 

 \begin{figure}[h!]
\centering
\includegraphics[trim = 0mm 0mm 0mm 0mm, clip,width=2.8in]{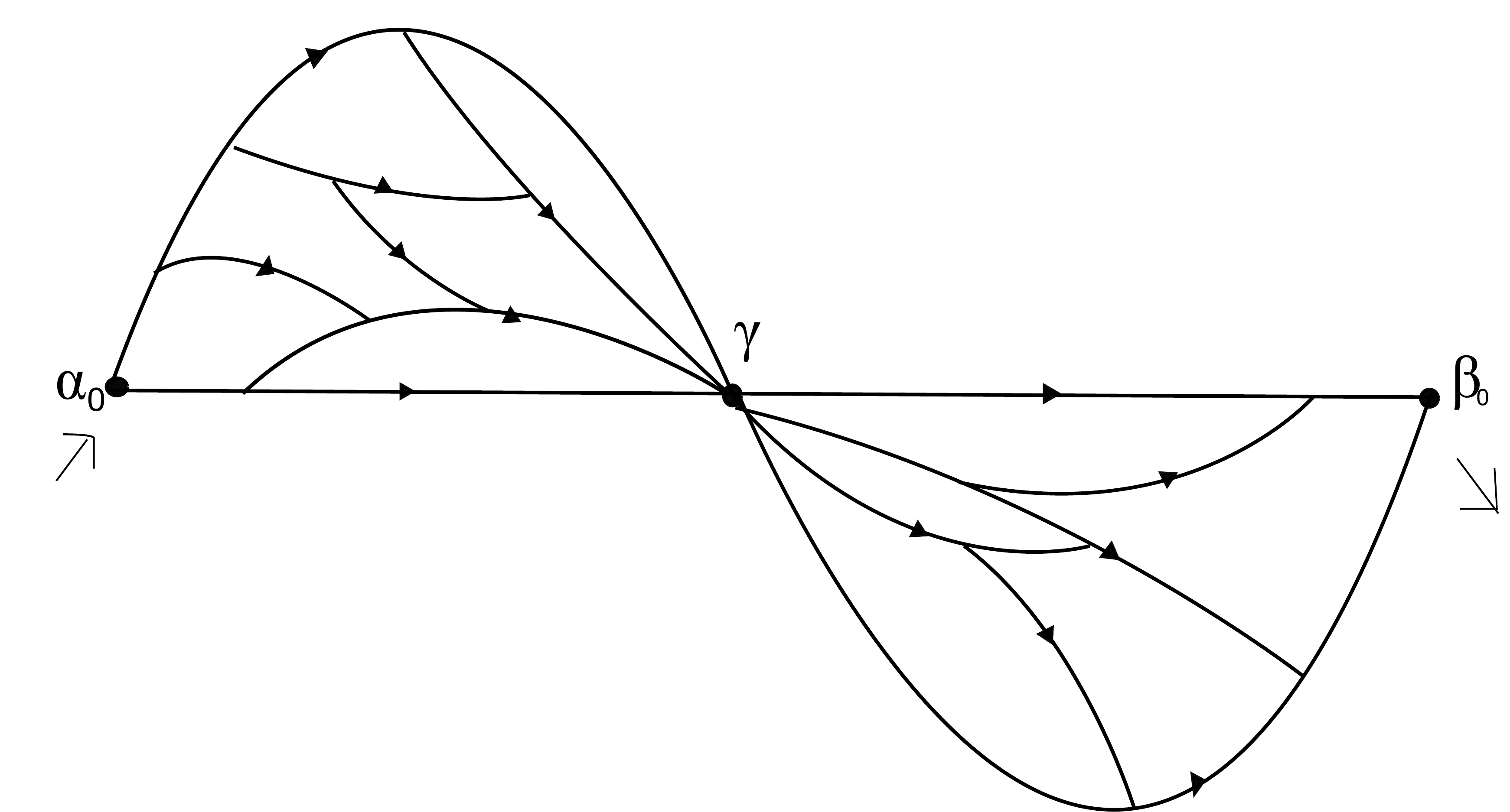}
\caption{$S_0$ with some transversals highlighted}
\label{2a}
\end{figure}

By hypothesis, the subgraphs of $S\Gamma(t_i)$, for $1\leq i\leq m$, generated by all the $R$-words are finite. It is possible that a transversal $t_i$, for some $i\in \{1,2,...,m\}$, may contain more than one unsaturated segment  passing through $\gamma$, because the $R$-words can overlap with themselves or with the other $R$-words.  We generate the subgraphs of $S\Gamma(t_i)$ by the $R$-words labeling an unsaturated segment passing through $\gamma$, and denote them by $\Gamma_j$, where $1\leq j\leq l$ for some $l\geq m$.

We sew on the finite subgraphs $\Gamma_1, \Gamma_2,...,,\Gamma_l$ to $S_0$ at the corresponding  transversals $t_i$'s (see Figure \ref{3a}). Consequently, we obtain a finite graph. We denote this new graph by $S_1$. No two edges fold together in $S_1$ by Lemma \ref{1}. Finally, we show that $S_1$ is $\mathscr{P}$-complete.  

\begin{figure}[h!]
\centering
\includegraphics[trim = 0mm 0mm 0mm 0mm, clip,width=2.8in]{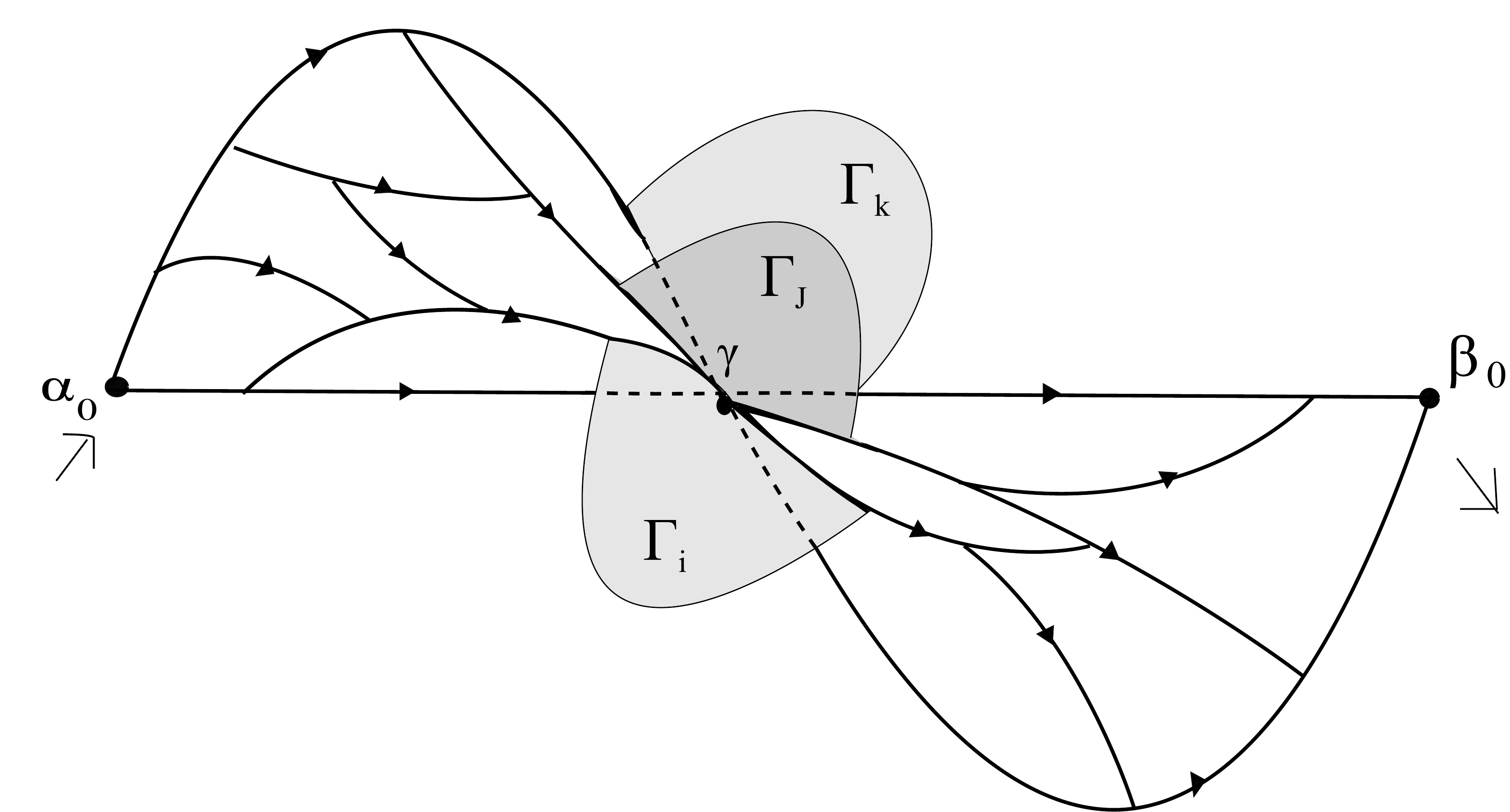}
\caption{$S_0$ with some $\Gamma_i$'s attahed}
\label{3a}
\end{figure}

To prove the $\mathscr{P}$-completeness of $S_1$ we need to show that $S_1$ does not contain any unsaturated segment, where elementary $\mathscr{P}$-expansions can be applied. So, we prove the following two claims. 

\begin{enumerate} 

\item  Every segment labeled by an $R$-word that either starts from a vertex of $S_0$ and terminates at a vertex of $\Gamma_i$, or vice versa, for some $i\in\{1,2,...,l\}$, is contained in $\Gamma_j$, for some $j$ $(1\leq j\neq i\leq l)$. 

\item 
Every segment labeled by an $R$-word that starts from a vertex of $\Gamma_i$ and terminates at a vertex of $\Gamma_j$, for $1\leq i\neq j\leq l$, is contained in $\Gamma_k$ for some $k\neq i,j$. 
\end{enumerate}

From claims $(1)$ and $(2)$, it follows that every segment of $S_1$ that is labeled by an $R$-word, is either entirely contained in $S_0$ or in $\Gamma_i$, for some $i\in\{1,2,...,l\}$. Hence, $S_1$ does not contain any unsaturated segment labeled by an $R$-word.

\textit{Proof of $(1)$}: Without loss of generality, we assume that we can find a segment labeled by an $R$-word that starts from a vertex $\gamma_0$ of $S_0$, passes through a vertex $\gamma_1$ of the transversal $t_i$, and terminates at a vertex $\gamma_2$ of $\Gamma_i$, for some $1\leq i\leq l$. Here $\gamma_1$ is the first vertex that is common between the transversal $t_i$ and the segment labeled by the $R$-word from $\gamma_0$ to $\gamma_2$. A dual argument can be used to show that every segment labeled by an $R$-word, that starts from a vertex of $\Gamma_i$  and terminates at a vertex of $S_0$, is contained in $\Gamma_j$, for $1\leq i\neq j\leq l$.

The vertex $\gamma_1$ lies either before the vertex $\gamma$, or at the vertex $\gamma$. If the vertex $\gamma_1$ lies after the vertex $\gamma$, then by Proposition $3(iii)$ of \cite{AD}, we can find a positively labeled segment from $\alpha_0$ to $\gamma_0$. We can extend this path to a transversal of $S_0$ by extending this path to $\gamma_1$, and proceeding along the transversal $t_i$ up to $\beta$.  This transversal does not pass through $\gamma$, which contradicts the fact that every transversal of $S_0$ passes through $\gamma$.

If the subsegment from $\gamma_1$ to $\gamma_2$, of the segment labeled by the $R$-word lies on the transversal $t_i$, then by Proposition $3(iii)$ of \cite{AD}, we can find a positively labeled segment from $\alpha_0$ to $\gamma_0$. We can extend this segment to a transversal of $S_0$ by proceeding along the segment from $\gamma_0$ to $\gamma_2$, and further proceeding along the transversal $t_i$.  We denote this transversal by $t$. If $\gamma_2$ lies before or at $\gamma$, then the segment from $\gamma_0$ to $\gamma_2$ lies in $S\Gamma(w_1)$. In this case, the segment from $\gamma_0$ to $\gamma_2$ cannot be unsaturated, as $S\Gamma(w_1)$ is $\mathscr{P}$-complete. If $\gamma_2$ lies after $\gamma$, that is, $\gamma_1$ lies at $\gamma$, or $\gamma$ lies between the vertices $\gamma_1$ and $\gamma_2$, then the transversal $t$ contains the segment from $\gamma_0$ to $\gamma_2$, that is labeled by an $R$-word, and passes through $\gamma$. So, the transversal $t$ is among one of the transversals  $t_1,t_2,...,t_m$. Hence, $t\equiv t_j$ where $1\leq i\neq j \leq m$. The subgraph $\Gamma_j$ generated by an $R$-word, that labels the segment from $\gamma_0$ to $\gamma_2$, and sewed onto the transversal $t_j$, contains the segment from $\gamma_0$ to $\gamma_2$. So, our claim holds in this case. 

If the vertex $\gamma_2$ does not lie on the transversal $t_i$, then by using Proposition $3(iii)$ of \cite{AD}, we can find a positively labeled segment in $\Gamma_i$ that starts from $\gamma_1$, passes through $\gamma_2$ and terminates at a vertex of the transversal $t_i$. We assume that this segment is labeled by $p'$, for some $p'\in X^+$.  We assume that the subsegment of $t_i$, from $\gamma_1$ to the terminal vertex of $p'$, is labeled by $p$, for some $p\in X^+$.  By Lemma $4$ of \cite{AD}, there exists an $S$-diagram whose boundary is labeled by the pair of words $(p,p')$. Again, by using Proposition $3(iv)$ of \cite{AD}, we can find a transversal $s$ of $S_0$ that passes through the subsegment from $\gamma_0$ to $\gamma_1$ of the segment labeled by the $R$-word, and the segment $p$. Clearly, the $S$-diagram corresponding to $(p,p')$ embeds in $S\Gamma(s)$. If the transversal $s$ does not contain an unsaturated segment labeled by an $R$-word, then $S\Gamma(s)$ embeds in $S_0$. This implies that the vertex $\gamma_2$ lies in $S_0$, which is a contradiction. Hence, $s$ contains an unsaturated segment labeled by an $R$-word passing through $\gamma$. Therefore, $s\equiv t_j$, for some $1\leq i\neq j \leq m$. The segment from $\gamma_0$ to $\gamma_2$, labeled by the $R$-word, is contained in some $\Gamma_j$, that was sewn onto the transversal $t_j$ and contains the $S$-diagram with boundary labels $(p, p')$.

\textit{Proof of $(2)$}: We assume that there exists a segment labeled by an $R$-word that starts from a vertex $\gamma_i$ of $\Gamma_i$ and terminates at a vertex $\gamma_j$ of $\Gamma_j$ for $i\neq j$. If this segment from $\gamma_i$ to $\gamma_j$ lies in $S_0$, such that it does not pass through the vertex $\gamma$, then this segment either lies in $S\Gamma(w_1)$ or in $S\Gamma(w_2)$. In this case, the segment from $\gamma_i$ to $\gamma_j$ cannot be unsaturated, as $S\Gamma(w_1)$ and $S\Gamma(w_2)$ both are $\mathscr{P}$-complete. 

If the segment from $\gamma_i$ to $\gamma_j$ lies in $S_0$, and passes through $\gamma$, then by Proposition $3(iv)$ of \cite{AD} we can extend this segment to a transversal labeled by $t$, for some $t\in X^+$ of $S_0$. The transversal $t$ contains a segment labeled by an $R$-word passing through $\gamma$,  therefore, $t\equiv t_k$, for some $k\in \{1,2,..., m\}$. So, the claim is true in this case, because the segment from $\gamma_i$ to $\gamma_j$ is contained in $\Gamma_k$ (where $\Gamma_k$ is the subgroup of $S\Gamma(t_k)$, generated by the $R$-word that labels the segment from $\gamma_i$ to $\gamma_j$).  

Now we assume that the segment from $\gamma_i$ to $\gamma_j$ does not lie in $S_0$, where the vertices  $\gamma_i$ and $\gamma_j$ can lie in $S_0$. We can find a vertex $\gamma_{t_i}$ on the transversal $t_i$, and a vertex $\gamma_{t_j}$ on the transversal $t_j$, such that the segments from $\gamma_{t_i}$ to $\gamma_i$, and from $\gamma_j$ to $\gamma_{t_j}$ are the shortest non-negatively labeled segments from $S_0$ to $\gamma_i$ and from $\gamma_j$ to $S_0$. If $\gamma_i$ or $\gamma_j$ lies in $S_0$, then $\gamma_i$ $(\gamma_j)$ and $\gamma_{t_i}$ $(\gamma_{t_j})$ represents the same vertex, otherwise, there will be positively labelled segment from $\gamma_{t_i}$ to $\gamma_i$ ($\gamma_{t_j}$ to $\gamma_j$). Obviously, none of the edges of the segments from $\gamma_{t_i}$ to $\gamma_i$, and from $\gamma_j$ to $\gamma_{t_j}$ lie in $S_0$, otherwise they will not be the shortest ones. We consider the positively labeled segment starting from $\gamma_{t_i}$, passing through the segment from $\gamma_i$ to $\gamma_j$ labeled by an $R$-word, and terminating at $\gamma_{t_j}$. We assume that this segment is labeled by $p'$, for some $p'\in X^+$. The vertex $\gamma_{t_i}$ lies in $S\Gamma(w_1)$ and the vertex $\gamma_{t_j}$ lies in $S\Gamma(w_2)$. So, $S_0$ contains a positively labeled segment from $\gamma_{t_i}$ to $\gamma_{t_j}$, passing through $\gamma$. We assume that the segment from $\gamma_{t_i}$ to $\gamma_{t_j}$, of $S_0$, is labeled by $p$, for some $p\in X^+$. By Lemma $4$ of \cite{AD}, there exists an $S$-diagram with boundary labeled by the pair of words $(p,p')$ that embeds in $S\Gamma(w)$. We denote this $S$-diagram by $\Pi$. 

The segment labeled by $p$ contains a subsegment labeled by an $R$-word passing through $\gamma$. Otherwise, the $S$-diagram $\Pi$, embeds in $S_0$. This leads us to a contradictory conclusion that the segment from $\gamma_i$ to $\gamma_j$, labeled by an $R$-word, lies in $S_0$.

The $S$-diagram $\Pi$ consists of only one simple component. If this $S$-diagram consists of more than one simple component, then there exists a vertex $\delta$ between $\gamma_{t_i}$ and $\gamma_{t_j}$, where the segments $p$ and $p'$ intersect with each other. We already observed that $p$ contains a subsegment labeled by  an $R$-word passing through $\gamma$. If $\delta$ lies anywhere between the initial and terminal vertex of the subsegment of $p$, that is labeled by an $R$-word and passes through $\gamma$, then it means that this subsegment of $p$ does not play any role in construction of the $S$-diagram $\Pi$. So, the $S$-diagram $\Pi$  embeds in $S_0$. This leads us to a contradiction that  the segment from $\gamma_i$ to $\gamma_j$, labeled by an $R$-word, lies in $S_0$. 

Now we show that $\delta$ cannot lie before or at the initial vertex of the subsegment of $p$, that is labeled by an $R$-word and passes through $\gamma$. A similar argument can be used to show that $\delta$ cannot lie at the terminal vertex or after the terminal vertex of the subsegment of $p$, that is labeled by an $R$-word and passes through $\gamma$.

If $\delta$ lies anywhere before the initial vertex or at the initial vertex of the subsegment of $p$, that is labeled by an $R$-word and passes through $ \gamma$, then $p$ can be factorized as $p\equiv p_1 p_2$, where $p_1$ labels the segment from $\gamma_{t_i}$ to $\delta$, and $p_2$ labels the segment from $\delta$ to $\gamma_{t_j}$. The segment $p_1$ lies in $S\Gamma(w_1)$. But, $S\Gamma(w_1)$ is $\mathscr{P}$-complete. Therefore, the subsegment of $p'$ from $\gamma_{t_i}$ to $\delta$, also lies in $S\Gamma(w_1)$, hence in $S_0$. This contradicts with the choice of the vertex $\gamma_{t_i}$. 

By Proposition $3(iii)$ of \cite{AD}, the segment $p$ can be extended to a transversal $t$ of $S_0$. Since this transversal $t$ contains a segment labeled by an $R$-word passing through $\gamma$, $t\equiv t_k$, for some $k\in\{1,2,...,m\}$. The $S$-diagram $\Pi$ consists of only one simple component. Therefore, it can be generated by some successive applications of elementary $\mathscr{P}$-expansions starting from the subsegment of $p$, that is labeled by an $R$-word and passes through $\gamma$. The subgraph $\Gamma_k$ generated by the subsegment of $p$, that is labeled by an $R$-word and passes through $\gamma$, is the largest subgraph of $S\Gamma(w)$ that can be generated by successive applications of elementary $\mathscr{P}$-expansions. So, the $S$-diagram $\Pi$ is contained in $\Gamma_k$.  This implies that the segment from $\gamma_i$ to $\gamma_j$ labeled by an $R$-word is also contained in $\Gamma_k$.

 \end{proof}
 
  \section{The Word Problem For Some classes of One relation Adian Inverse Semigroups}


In this paper, we only focus on those one relation Adian inverse semigroups, in which no $R$-word is a subword of the other $R$-word. 

\begin{lemma} \label{L2}Let $M=Inv\langle X|u=v\rangle$ be an Adian inverse semigroup. If none of the $R$-words is a subword of the other $R$-word, and no $R$-word overlaps with itself or with the other $R$-word, then $S\Gamma(t)$, for all $t\in X^+$, is finite. 
\end{lemma}

\begin{proof}
Let $t$ be a positive word. If $t$ does not contain any $R$-word as its subword, then there is nothing to prove, because $S\Gamma(t)$ is finite in this case.

We assume that $t$ contains $n$ number of (not necessarily distinct) $R$-words, for some $n \in \mathbb{N}$. We construct the linear automaton of $t$, denoted by $(\alpha_0,\Gamma_0(t), \beta_0)$.  We apply full $\mathscr{P}$-expansion on $(\alpha_0,\Gamma_0(t), \beta_0)$ and obtain $(\alpha_1,\Gamma_1(t), \beta_1)$. Since none of the $R$-word is subword of the other $R$-word, we cannot find any unsaturated sub-segment labeled by an $R$-word  of the newly attached segments of $(\alpha_1,\Gamma_1(t), \beta_1)$. We know from Proposition 3(iii) of \cite{AD}, that $(\alpha_1,\Gamma_1(t), \beta_1)$ is a directional graph in the sense that every vertex of $(\alpha_1,\Gamma_1(t), \beta_1)$ lies on a transversal from $\alpha_1$ to $\beta_1$.  We cannot find an unsaturated segment labeled by an $R$-word that either starts from a vertex which lies before the initial vertex of a newly attached segment and terminates at an interior vertex of the newly attached segment, or that starts from an interior vertex of a newly attached segment and terminates at a vertex which lies after the terminal vertex of the newly attached segment, as none of the $R$-words overlap with themselves or with the other $R$-word. So, $(\alpha_1,\Gamma_1(t), \beta_1)$ cannot be expanded further by applying full $\mathscr{P}$-expansion.  Hence, the underlying finite graph of $(\alpha_1,\Gamma_1(t), \beta_1)$ is $S\Gamma(t)$. 
\end{proof}

The one relation Adian presentations with no $R$-word being a subword of the other $R$-word, as well as some overlap between the $R$-words, can be distributed into four classes. 

\begin{proposition}\label{types} Let $\langle X|u=v\rangle$ be a positive presentation, such that no $R$-word is a subword of the other $R$-word. There are four different types of overlaps possible in the presentation $\langle X|u=v\rangle$.
\begin{enumerate}

  \item One of the $R$-word has same prefix and suffix. The other $R$-word doesn't overlap with itself and there is no overlap between the two $R$-words. 
   
 \item Both the $R$-words overlap with themselves and they do not overlap with each other. 
  
 \item  Prefix of an $R$-word is suffix of the other $R$-word, and no suffix of the former $R$-word is a prefix of  the latter $R$-word.
  
\item A prefix of one $R$-word is a suffix of the other $R$-word, and a suffix of the former $R$-word is a prefix of the later $R$-word.

 \end{enumerate}

\end{proposition}

\begin{proof} There are only two $R$-words $u$ and $v$ in the presentation $\langle X|u=v\rangle$. So, there can be only two possibilities of an $R$-word overlapping with itself. That is, either one of the $R$-words overlaps with itself or both of the $R$-words overlap with themselves. Both of these possibilities are considered in cases $(1)$ and $(2)$ of the above statement. Obviously, cases $(1)$ and $(2)$ cannot occur simultaneously. 

There are only two possibilities of an $R$-word overlapping with the other $R$-word in an Adian presentation of the form $\langle X|u=v\rangle$. That is, either a prefix of an $R$-word is a suffix of the other $R$-word, and no suffix of the former $R$-word is a prefix of  the latter $R$-word, or a prefix of one $R$-word is the suffix of the other $R$-word and a suffix of the former $R$-word is the prefix of the latter $R$-word. Both of these possibilities are considered in cases $(3)$ and $(4)$ of the above statement. Obviously, cases $(3)$ and $(4)$ cannot occur simultaneously. 

In case $(3)$, it is also possible that one or both of the $R$-words overlap with themselves as well. But, these possibilities are considered as sub-cases of $(3)$. Also in case $(4)$, it is possible that one or both of the $R$-words overlap with themselves. But, these possibilities are considered as sub-cases of $(4)$. 
\end{proof}

We say that an Adian inverse semigroup $M=\langle X|u=v\rangle$ is of type $(1)$, when the presentation $\langle X|u=v\rangle$ is of type $(1)$ of Proposition \ref{types}.

\begin{lemma}\label{L3} Let $M=Inv\langle X|u=v\rangle$ be an Adian inverse semigroup of type $(1)$.  Then every subgraph of $S\Gamma(w)$, for all $w\in X^+$,  generated by an $R$-word is finite. 
\end{lemma}

\begin{proof} Without loss of generality, we assume that $u$ overlaps with itself and $v$ neither overlaps with itself nor with $u$. 

Obviously, for any $w\in X^+$ that doesn't contain any $R$-word as its subword, $S\Gamma(w)$ is finite. So, we consider an arbitrary positive word $w$, such that $w$ contains some $R$-words as its subwords. We construct the linear automaton of $w$, $(\alpha_0,\Gamma_0(w),\beta_0)$. 

If $w$ contains $u$ as its subword, we apply the elementary $\mathscr{P}$-expansion on the segment labeled by $u$ of $(\alpha_0,\Gamma_0(w),\beta_0)$ by sewing on a path labeled by $v$ from the initial vertex to the terminal vertex of the segment labeled by $u$. Since $v$ neither overlaps with itself nor with $u$, the first application of elementary $\mathscr{P}$-expansion does not create any new unsaturated segment labeled by an $R$-word. So, the subgraph of $S\Gamma(w)$ generated by $u$ stops to grow at the first iterative step. Hence, the subgraph of $S\Gamma(w)$ generated by $u$ is finite.

If $w$ contains $v$ as its subword, we apply elementary $\mathscr{P}$-expansion  on the segment labeled by $v$ of $(\alpha_0,\Gamma_0(w),\beta_0)$ by sewing on a path labeled by $u$ from the initial vertex to the terminal vertex of the segment labeled by $v$. Since $u$ overlaps with itself, we may find a finite number of unsaturated segments labeled by $u$ starting from a vertex of $(\alpha_0,\Gamma_0(w),\beta_0)$ that lies before the initial vertex of the segment labeled by $v$ and terminating at an interior vertex of the segment labeled by $u$ or starting from an interior vertex of the segment labeled by $u$ and terminating at a vertex of $(\alpha_0,\Gamma_0(w),\beta_0)$ that lies after the terminal vertex of the segment labeled by $v$. We apply elementary $\mathscr{P}$-expansions on all of these new unsaturated segments by sewing on segments labeled by $v$. Since $v$ does not overlap with itself or with $u$, the second applications of elementary $\mathscr{P}$-expansions do not create any new unsaturated segments labeled by some $R$-words. Hence, the subgraph of $S\Gamma(w)$ generated by $v$ remains finite. 

\end{proof}

We say that an Adian inverse semigroup $M=\langle X|u=v\rangle$ is of type $(2)$ when the presentation $\langle X|u=v\rangle$ is of type $(2)$ of Proposition \ref{types}.

\begin{proposition}\label{P1} Let $M=Inv\langle X|u=v\rangle$ be an Adian inverse semigroup of type $(2)$. Then all the segments sewn on in an iterative step of the construction of a subgraph of $S\Gamma(w)$, for all $w\in X^+$, generated by an $R$-word, are labeled by the same $R$-word. 
\end{proposition} 
\begin{proof}  If $w\in X^+$ doesn't contain any $R$-word as it's subword, then $S\Gamma(w)$ is finite. So, there is nothing to prove in this case. 

Let $w\in X^+$ be an arbitrary word that contains some $R$-words as its ssubwords. To prove the above statement we use mathematical induction on the number of iterative steps involved in the construction of a subgraph of $S\Gamma(w)$, generated by an $R$-word.  Without loss of generality, we assume that $v$ is a subword of $w$.  We construct the linear automation of $w$, $(\alpha_0,\Gamma_0(w), \beta_0)$. At the first iterative step, we sew on a segment labeled by $u$ from the initial vertex to the terminal vertex of the segment labeled by $v$ of $(\alpha_0,\Gamma_0(w), \beta_0)$. So, the above statement is obviously true for the first iterative step. 

We assume that the above statement is true for some $n\in\mathbb{N}$, where $n\geq2$. That is, all the segments that are sewn on at the $n$-th iterative step are labeled by the same $R$-word. Since both the $R$-words overlap with themselves and there is no overlap between them, if some of the segments sewn on at the $n$-th iterative step create some new unsaturated segments, they must be labeled by the same $R$-word that labels all the segments sewn on at the $n$-th iterative step. So, all the segments sewn on at the $(n+1)$-st iterative step, are labeled by the other $R$-word.  
\end{proof}

We can distribute the Adian presentations of type $(2)$ into the following two categories.

$(a)$. No proper subword of one $R$-word is a proper subword of the other $R$-word.

$(b)$. A proper subword of one $R$-word is a proper subword of the other $R$-word. 

For instance, the presentation $\langle a,b,c|aba=c^2\rangle$ is of type $(2a)$, and the presentation $\langle a,|aba=b^3\rangle$ is of type $(2b)$ as $b$ is common between both the $R$-words.





\begin{proposition} Let $M=Inv\langle X|u=v \rangle$ be an Adian inverse semigroup of type $(2a)$.  If every third generation region of a subgraph of $S\Gamma(w)$,  for all $w\in X^+$, generated by an $R$-word uses an edge of the linear automaton of $w$, then every higher generation region of the subgraph uses an edge of the linear automaton of $w$. 
\end{proposition}

\begin{proof} If $w\in X^+$ does not contain any $R$-word as its subword, then $S\Gamma(w)$ is finite. So, there is nothing to prove in this case.

Let $w\in X^+$ be an arbitrary word that contains some $R$-words as its subwords. We apply mathematical induction on the number of iterative steps involved in the construction of a subgraph of $S\Gamma(w)$, generated by an $R$-word. 

 We assume that every third generation region of the subgraph of $S\Gamma(w)$, generated by an $R$-word, uses an edge of the linear automaton of $w$, denoted by $(\alpha_0, \Gamma_0(w), \beta_0)$. For the base case of our induction, we show that every fourth generation region of the subgraph uses an edge of $(\alpha_0, \Gamma_0(w), \beta_0)$.
 
 Since every third generation region uses an edge of $(\alpha_0, \Gamma_0(w), \beta_0)$, either the initial vertex or the terminal vertex of all the segments sewed on at the third iterative step, lies on $(\alpha_0, \Gamma_0(w), \beta_0)$. Without loss of generality, we assume that the initial vertex of a segment, labeled by an $R$-word, and  sewed on at the third iterative step, lies on $(\alpha_0, \Gamma_0(w), \beta_0)$. We denote this segment by $s_3$. 
 Since both the $R$-words overlap with themselves and there is no overlap between them, if the third iterative step creates some new unsaturated segments, then they are labeled by the same $R$-word as $s_3$.
 
  The terminal vertex of $s_3$ lies on a segment that was sewn on at the second iterative step and labeled by the other $R$-word. If $s_3$ creates some new unsaturated segments, then these new unsaturated segments either start from a vertex of $(\alpha_0, \Gamma_0(w), \beta_0)$ which lies before the initial vertex of $s_3$ and terminates at an interior vertex of $s_3$, or they start from an interior vertex of $s_3$ and passes through a segment labeled by the other $R$-word and sewed on at the second iterative step. The later case is not possible, as no proper subword of one $R$-word is a proper subword of the other $R$-word. Therefore, if a fourth generation region exists, then it also uses an edge of $(\alpha_0, \Gamma_0(w), \beta_0)$.  

We assume that every $k$-th generation region uses an edge of $(\alpha_0, \Gamma_0(w), \beta_0)$, for some $k>4$. Then either the initial vertex or the terminal vertex of every segment that is sewed on at the $k$-th iterative step lies at $(\alpha_0, \Gamma_0(w), \beta_0)$.
Without loss of generality, we assume that the initial vertex of a segment labeled by an $R$-word and sewed on at the $k$-th iterative step lies on $(\alpha_0, \Gamma_0(w)$ $, \beta_0)$. We denote this segment by $s_k$. 
Since both the $R$-words overlap with themselves and there is no overlap between both of them, if the $k$-th iterative step creates some new unsaturated segments, then they will be labeled by the same $R$-word as $s_k$. Moreover, these new unsaturated segments either start from a vertex of $(\alpha_0, \Gamma_0(w), \beta_0)$ which lies before the initial vertex of $s_k$ and terminates at an interior vertex of $s_k$ or they start from an interior vertex of $s_k$ and passes through the segment labeled by the other $R$-word that was sewed on at the $(k-1)$-st iterative step. The latter case is not possible as no proper subword of one $R$-word is a proper subword of the other $R$-word. Therefore, if the $(k+1)$-st generation region exists, then it also uses an edge of $(\alpha_0, \Gamma_0(w), \beta_0)$.  

\end{proof}

We call a vertex $\delta$ of $S\Gamma(w)$, for some $w\in X^+$, to be a \textit{special vertex}, if $\delta$ is the terminal vertex of a segment labeled by an $R$-word and the initial vertex of a segment labeled by the same $R$-word. 

\begin{proposition}\label{P4} If a word $t\in X^+$ overlaps with itself, then $t$ is in one of the following forms.

\begin{enumerate}
\item $t\equiv x^n$, where $x\in X^+$, and $n\geq 2$. 

\item $t\equiv xsx$, where $x,s\in X^+$, where $x$ is the maximal prefix of $t$ that is also a suffix of $t$. 

\item $t\equiv xsxsx$, where $x,s\in X^+$, where $xsx$ is the maximal prefix of $t$ that is also a suffix of $t$
\end{enumerate}

\end{proposition} 

\begin{proof}  
If $t\equiv x^n$, for some $x\in X^+$ and $n\geq 2$, then clearly $t$ overlaps with itself. If $t\not\equiv x^n$, for some $x\in X^+$, then there are two possibilities. 

\begin{enumerate}
\item The maximal prefix $x\in X^+,$ of $t$, that is also a suffix of $t$, doesn't overlap with itself in $t$. 

\item The maximal prefix $y\in X^+$, of $t,$ that is also a suffix of $t$, overlaps with itself in $t$. 
\end{enumerate} 

In the first case, $t\equiv xsx$, for some $x,s\in X^+$. 

In the second case, since $y$ overlaps with itself in $t$, a prefix of $y$ is also a suffix of $y$. Let $x\in X^+$ be the maximal prefix of $y$ that is also a suffix of $y$. Then $y\equiv xsx$, for some $x,s\in X^+$. We observe that $s\not\equiv \epsilon$ or $s\not\equiv x$. Otherwise, $t$ will be of the form $x^n$. Hence, $t\equiv xsxsx$, for some $x,s\in X^+$. 

\end{proof}

\begin{proposition}\label{P5} Let $M=Inv\langle X|u=v \rangle$ be an Adian inverse semigroup of type $(2a)$. A segment of $S\Gamma(w)$, for some $w\in X^+$,  labeled by $u^2$ contains a subsegment labeled by $u$ passing through the special vertex of the segment $u^2$ if and only if $u\equiv x^n$, for some $x\in X^+$ and $n\geq 2$. 
\end{proposition}

\begin{proof} The converse of the above statement is obvious. Therefore, we only prove the direct statement.  

Let a segment of $S\Gamma(w)$, for some $w\in X^+$,  labeled by $u^2$ contains a subsegment labeled by $u$ passing through the special vertex, and $u\not\equiv x^n$, for some $x\in X^+$ and $n\geq 2$. Then by Proposition \ref{P4}, either $u\equiv xsx$, for some $x,s\in X^+$, where $x$ is the maximal prefix of $u$ that is also a suffix of $u$, or $u\equiv xsxsx$, for some $x,s\in X^+$, where $xsx$ is the maximal prefix of $u$ that is also a suffix of $u$. 

\textit{Case 1}; We assume that $u\equiv xsx$, for some $x,s\in X^+$, where $x$ is the maximal prefix of $u$ that is also a suffix of $u$. By hypothesis, there exists a subsegment of the segment $u^2$, that is labeled by $u$, and passing through the special vertex of  the segment $u^2$. If $u$ can be read starting from the  initial vertex of $x$ (where $x$ is the suffix of the first $u$ in the segment $u^2$), then $x\equiv s$. This implies that $u\equiv x^3$, which contradicts with our assumption that $u\not\equiv x^n$, for some $x\in X^+$ and $n\geq 2$. 

If $u$ can be read starting from a vertex that lies before the initial vertex of $x$ (where $x$ is the suffix of the first $u$ in the segment $u^2$), then a longer prefix of $u$ will be a suffix of $u$. This contradicts with the maximality of $x$. 
 
If $u$ can be read starting from a vertex that lies between the initial and terminal vertices of $x$ (where $x$ is the suffix of the first $u$ in the segment $u^2$), then again a longer prefix of $u$ will be a suffix of $u$. This contradicts with the maximality of $x$. 

\textit{Case 2}: We assume that $u\equiv xsxsx$, for some $x,s\in X^+$, where $xsx$ is the maximal prefix of $u$ that is also a suffix of $u$. By hypothesis, there exists a subsegment of the segment $u^2$, that is labeled by $u$, and passing through the special vertex of the segment $u^2$. If this subsegment starts from the initial vertex of $y$ (where $y\equiv xsx$ a  suffix of the first $u$ in the segment $u^2$), this is only possible when $x\equiv s$.  This implies that $u\equiv x^5$, for some $x\in X^+$, which contradicts with our assumption that  $u\not\equiv x^n$, for some $x\in X^+$ and $n\geq 2$. 

If the subsegment labeled by $u$ starts from a vertex that lies before the initial vertex of $y$, then a longer prefix of $u$ will be a suffix of $u$, which contradicts the maximality of of $xsx$. 

Similarly, if the subsegment labeled by $u$ starts from a vertex that lies between the initial and terminal vertices of $y$, then a longer prefix of $u$ will be a suffix of $u$, which contradicts the maximality of $xsx$. 

\end{proof}

We call a third generation region of a subgraph of $S\Gamma(w)$, for some $w\in X^+$, generated by an $R$-word,  that does not use an edge of the linear automaton of $w$, to be a \textit{special region}. 

\begin{proposition}\label{P6}
Let $M=Inv\langle X|u=v \rangle$ be an Adian inverse semigroup of type $(2a)$. A subgraph of $S\Gamma(w)$, for some $w\in X^+$, generated by an $R$-word, contains a special region if an only if $u\equiv x^n$ and $v\equiv y^m$, for some $x, y\in X^+$ and $n,m\geq 2$. 

\end{proposition}

\begin{proof} The converse of the above proposition is easy to see. Let $u\equiv x^n$ and $v\equiv y^m$, for some $x, y\in X^+$ and $n,m\geq 2$. Then $S\Gamma(y^{p_1}uy^{q_1})$, where $p_1+q_1=n$, and $S\Gamma(x^{p_2}vx^{q_2})$, where $p_2+q_2=m$, contain a third generation region that does not use an edge of linear automatons of $y^{p_1}uy^{q_1}$ and $x^{p_2}vx^{q_2}$, respectively. 

Now we prove the direct statement. We assume that a subgraph os $S\Gamma(w)$, for some $w\in X^+$, generated by an $R$-word, contains a special region. Without loss of generality, we assume that one side of this special region that lies on the boundary of second generation regions of the subgraph, is labeled by $u$. Since $\langle X|u=v\rangle$ is a presentation of type $(2a)$, one side of the special region labeled by $u$ lies on a segment labeled by $u^2$  of the subgraph of $S\Gamma(w)$, generated by an $R$-word  (see Figure \ref{4a}). The segment labeled by $u^2$ contains a special vertex and the one side of the special region labeled by $u$ is passing through this special vertex. So, by Proposition \ref{P5}, $u\equiv x^n$, for some $x\in X^+$ and $n\geq 2$. The segment labeled by $u^2$ can only exist, if there is a segment labeled by $v^2$ after the first iterative step. The segment labeled by $v^2$ contains a special vertex that lies on the segment labeled by an $R$-word, and sewed on at the first iterative step. Since $\langle X|u=v\rangle$ is an Adian presentation of type $(2a)$, the segment sewed on at the first iterative step is labeled by $v$. The segment labeled by $v^2$ contains a segment labeled by $v$ that is passing through the special vertex. By Proposition \ref{P5}, $v\equiv y^m$, for some $y\in X^+$ and $m\geq 2$. 

\begin{figure}[h!]
\centering
\includegraphics[trim = 0mm 0mm 0mm 0mm, clip,width=2.8in]{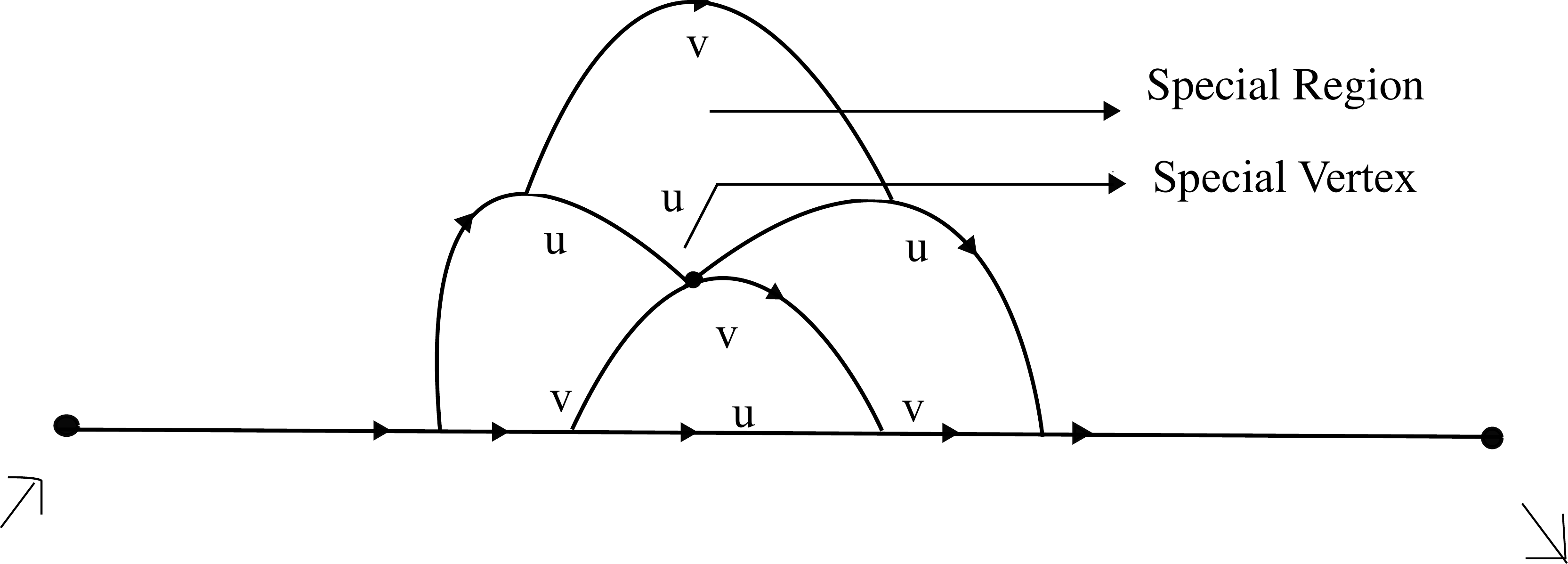}
\caption{}
\label{4a}
\end{figure}
 
\end{proof}

An Adian inverse semigroup $M=Inv\langle X|u=v\rangle$ is of type $(3)$ when the presentation $\langle X|u=v\rangle$ is of type $(3)$ of Proposition  \ref{types}.

\begin{proposition}\label{P8} Let $M=Inv\langle X|u=v\rangle$ be an Adian inverse semigroup of type $(3)$, such that no $R$-word overlaps with itself. Then there is no third generation region in a subgraph  of $w$ generated by an $R$-word, for all $w\in X^+$,  that does not use an edge of linear automaton of $w$. 
\end{proposition}

\begin{proof} Without loss of generality, we assume that $u\equiv xy$ and $v=yz$, where $x,y,z\in X^+$, and $y$ is the maximal prefix of $v$ that is also a suffix of $u$. We also assume that $u$ is a subword of $w$, for some $w\in X^+$. We show that the subgraph of $S\Gamma(w)$ generated by $u$, cannot contain a third generation region that does not use an edge of the linear automaton of $w$. Since $w$ is any arbitrary positive word, the above proposition is true for any positive word that contains $u$ as its subword. 

We construct the linear automaton of $w$. We sew on a segment labeled $v$ from the initial vertex to the terminal vertex of the segment labeled by $u$. This creates the first generation region of the subgraph of $S\Gamma(w)$ generated by $u$. Since $u$ and $v$ do not overlap with themselves, and no prefix of $u$ is a suffix of $v$, the second generation region of the subgraph generated by $u$ exists only if the word $u$ is followed by $x$ in $w$. So, we assume that the subword $u$ is followed by $x$ in $w$. Now, we can find an unsaturated segment labeled by $u$ starting from the initial vertex of $x$ and terminating at an interior vertex of the segment labeled by $v$ that was sewn on at the first step. We sew on a segment labeled $v$ from the initial vertex to the terminal vertex of this unsaturated segment, and create the second generation region of the subgraph generated by $u$. 

If we assume that the subgraph of $S\Gamma(w)$ generated by $u$ contains a third generation region that does not use an edge of the linear automaton of $w$, then there exist an unsaturated segment  labeled by an $R$-word that starts from an interior vertex of the segment labeled by $v$ (that was sewn at the second step) and terminates at a vertex of the segment labeled by $v$ (that was sewn at the first step). If this unsaturated segment is labeled by $v$, then this implies that $v$ overlaps with itself, which is a contradiction to the fact that no $R$-overlaps with itself. If this unsaturated segment is labeled by $u$, then it implies that a suffix of $v$ is also a prefix of $u$, which is also a contradiction. 

A dual argument can be used to show that a subgraph of $S\Gamma(w)$, for all $w\in X^+$, generated by $v$ cannot contain a third generation region that does not use an edge of the linear automation of $w$.

\end{proof}

\begin{proposition}\label{P9} Let $M=Inv\langle X|u=v\rangle$ be an Adian inverse semigroup of type $(3)$, such that no $R$-word overlaps with itself. Then every region of a subgraph of $S\Gamma(w)$, for all $w\in X^+$, generated by an $R$-word uses an edge of the linear automaton of $w$. 
\end{proposition} 

\begin{proof} Without loss of generality, we assume that $u\equiv xy$ and $v=yz$, where $x,y,z\in X^+$, and $y$ is the maximal prefix of $v$ that is also a suffix of $u$. Let $w$ be an arbitrary positive word that contains $u$ as its subword. We prove the above statement for a subgraph of $S\Gamma(w)$ generated by $u$, a similar argument can be used to prove the above statement for a subgraph generated by $v$. 

We use mathematical induction on the number of iterative steps involved in the construction of the subgraph generated by $u$. We have already seen in Proposition \ref{P8} that every region of the subgraph generated by $u$ up to the third generation uses an edge of the linear automaton of $w$. Which proves the base case for the mathematical induction (that is, every third generation region of the subgraph generated by $u$ uses an edged of the linear automaton of $w$). 

We assume that every region of the subgraph generated by $u$ up to $k$-th iterative step uses an edge of the linear automaton of $w$, for some $k\geq 3$. We show that the $(k+1)$-generation region of the subgraph generated by $u$ also uses an edge of linear automaton of $w$. 

We observe in the proof of Proposition \ref{P8} that we only sew on segment labeled by $v$ in the first three iterative steps involved in the construction the subgraph generated by $u$. So if we continue in the same manner up to the $k$-th iterative step, then the segment labeled by $v$ is sewn on at the $k$-th iterative step to form the $k$-th generation region of the subgraph generated by $u$. The initial vertex of this segment lies at the linear automaton of $w$ and the terminal vertex of this segment lies on the segment labeled by $v$ sewed on at the $(k-1)$-st iterative step. 

We cannot find an unsaturated segment labeled by an $R$-word that starts from an interior vertex of the segment labeled by $v$ (sewed on at the $k$-th iterative step) and terminating anywhere after passing through the terminal vertex of this segment labelled by $v$. Because if such a segment exists and it is labeled by $v$, then it means that $v$ overlaps with itself. This contradicts the fact that no $R$-word overlaps with itself. If such a segment is labeled by $u$ then it is only possible when a prefix of $u$ is also a suffix of $v$, which again a contradiction. So if the $k$-th iterative step creates a new unsaturated segment, then this segment will start from a vertex of the linear automaton of $w$ and after passing through the initial vertex of the segment $v$ (sewed on at the $k$-th iterative step) terminate at an interior vertex of this segment $v$. Hence the $(k+1)$-st generation region also uses an edge of the linear automaton of $w$. 

\end{proof}

\begin{theorem}\label{} Let $M=Inv\langle X|u=v\rangle$ be an Adian inverse semigroup, such that no $R$-word is a subword of the other $R$-word, and the relation $(u,v)$ is in one of the following forms:
\begin{enumerate} 

\item No $R$-word overlaps with itself or with the other $R$-word.

\item One of the $R$-words overlaps with itself, and the other $R$-word neither overlaps with itself nor with the former $R$-word. 

\item Both $R$-words overlap with themselves, there is no overlap between both the $R$-words, and at least one of the $R$-words is not of the form $x^n$, for some $x\in X^+$ and $n\geq 2$.  

\item Prefix of one $R$-word is a suffix of the other $R$-word, no suffix of the former $R$-word is a prefix of  the latter $R$-word, and no $R$-word overlaps with itself. 
\end{enumerate}

Then the word problem is decidable for $M$. \end{theorem}

\begin{proof} If the relation $(u,v)$ is of the form $(1)$, then   by Lemma \ref{L2} $S\Gamma(w)$, for all $w\in X^+$, is finite. Therefore, by Theorem \ref{PMT} $S\Gamma(w)$, for all $w\in (X\cup X^{-1})^*$, is finite, which implies that the word problem is decidable for $M$. 

If the relation $(u,v)$ is of the form $(2)$, then by Lemma \ref{L3} every subgraph of $S\Gamma(w)$, for all $w\in X^+$, generated by an $R$-word is finite. Therefore, by Theorem \ref{MT}, $S\Gamma(w)$, for all $w\in X^+$, is finite. So, by Theorem \ref{PMT}, $S\Gamma(w)$, for all $w\in (X\cup X^{-1})^*$, is finite. Hence, the word problem is decidable for $M$.

If the relation $(u,v)$ is of the form $(3)$, then by Proposition \ref{P6}, none of the subgraphs of $S\Gamma(w)$, for all $w\in X^+$, generated by an $R$-word, contains a special region. Therefore, for any arbitrary $w\in X^+$, every region of a subgraph of $S\Gamma(w)$, generated by an $R$-word, uses an edge of the linear automaton of $w$. Since $w$ is a word of finite length, every subgraph of $S\Gamma(w)$, generated by an $R$-word, is finite. Since $w\in X^+$ is arbitrary, every subgraph of $S\Gamma(w)$, for all $w\in X^+$, is finite. It follows by Theorem \ref{MT} that $S\Gamma(w)$, for all $w\in X^+$, is finite. By Theorem \ref{PMT}, $S\Gamma(w)$, for all $w\in (X\cup X^{-1})^*$, is finite. So, the word problem is decidable for $M$. 

If the relation $(u,v)$ is of the form $(4)$, then by Proposition \ref{P9} every region of a subgraph of $S\Gamma(w)$, for all $w\in X^+$, generated by an $R$-word uses an edge of the linear automaton of $w$. Since $w$ is a word of finite length, every subgraph of $S\Gamma(w)$, generated by an $R$-word, is finite. Since $w\in X^+$ is arbitrary, every subgraph of $S\Gamma(w)$, for all $w\in X^+$, is finite. It follows by Theorem \ref{MT} that $S\Gamma(w)$, for all $w\in X^+$, is finite. By Theorem \ref{PMT}, $S\Gamma(w)$, for all $w\in (X\cup X^{-1})^*$, is finite. So, the word problem is decidable for $M$.

\end{proof}

 
 \bibliographystyle{plain}

\end{document}